\documentclass{amsart}

\usepackage[T1]{fontenc}
\usepackage{times}

\usepackage{amssymb,epsfig,verbatim}
\usepackage{calligra,mathrsfs}

\theoremstyle{plain}
\usepackage[T1]{fontenc}
\usepackage[utf8]{inputenc}
\usepackage[english]{babel}

\usepackage{amsthm}
\usepackage{graphics}
\usepackage{marginnote}
\usepackage{amsmath}
\usepackage{amstext}
\usepackage{multirow}
\usepackage{enumerate}
\usepackage{enumitem}
\usepackage{hyperref}
\usepackage{graphicx}
\usepackage{amsmath}
\usepackage{amssymb}
\usepackage{amsthm}
\usepackage{amstext}
\usepackage{epstopdf}
\usepackage{mathrsfs}
\usepackage{amscd}
\usepackage{tikz}
\usetikzlibrary{cd}
\AtBeginDocument{%
   \def\MR#1{}
}

\usepackage{float}

\usetikzlibrary{matrix}

\usepackage{todonotes}

\definecolor{approvalBlue}{rgb}{0.95,0.95,1.0}

\newtheorem{thm}{Theorem}[section]
\newtheorem{lemma}[thm]{Lemma}
\newtheorem{prop}[thm]{Proposition}

\newtheorem{cor}[thm]{Corollary}

\newtheorem{THM}{Theorem}

\newtheorem{ex}[thm]{Example}

\theoremstyle{remark}
\newtheorem{remark}[thm]{Remark}

\newcommand{\mb}{\mathbb}

\newcommand{\mc}{\mathcal}

\newcommand{\C}{\mb C}

\newcommand{\F}{\mc F}
\newcommand{\G}{\mc G}

\newcommand\restr[2]{{
  \left.\kern-\nulldelimiterspace 
  #1 
  \vphantom{\big|} 
  \right|_{#2} 
  }}

\tikzset{
    table/.style={
       matrix of nodes,
        row sep=-\pgflinewidth,
        column sep=-\pgflinewidth,
        nodes={
            rectangle,
            draw=black,
            align=center
        },
        minimum height=1.5em,
        text depth=0.5ex,
        text height=2ex,
        nodes in empty cells,
        every even row/.style={
            nodes={fill=gray!20}
        },
        column 1/.style={
            nodes={text width=3.0cm,align=center,font=\small}
        },
        column 2/.style={
            nodes={text width=3.4cm,align=center},font=\small},
        column 3/.style={
            nodes={text width=5.0cm,align=center},font=\small},
        row 1/.style={
            nodes={ font=\bfseries
            }
        }
   }
}

\newcommand{\Dist}[4][1]{\textsf{Dist}^{#1}_{#2}(\mathbb P_{#4}^{#3})}
\newcommand{\Fol}[4][1]{\textsf{Fol}^{#1}_{#2}(\mathbb P_{#4}^{#3})}
\newcommand{\Exc}[2]{\textsf{Exc}(\mathbb P_{#2}^{#1})}
\newcommand{\Closed}[3]{\textsf{Cl}_{#1}(\mathbb P_{#3}^{#2})}
\newcommand{\Lin}[3]{\textsf{Lin}_{#1}(\mathbb P_{#3}^{#2})}
\newcommand{\QLin}[3]{\textsf{NLin}_{#1}(\mathbb P_{#3}^{#2})}
\newcommand{\Log}[3]{\textsf{Log}_{(#1)}(\mathbb P_{#3}^{#2})}

\newcommand{\Rational}[1]{\mathcal M_{#1}}
\newcommand{\normal}[1]{#1^n}

\DeclareMathOperator{\codim}{codim}
\DeclareMathOperator{\sing}{sing}

\DeclareMathOperator{\Aut}{Aut}

\DeclareMathOperator{\Ann}{Ann}

\DeclareMathOperator{\Res}{Res}

\DeclareMathOperator{\Frob}{Frob}

\DeclareMathOperator{\HomSheaf}{\mathscr{H}\text{\kern -3pt {\calligra\large om}}\,}

\newcommand{\ie}{{\it{i.e., }}}

\newcommand{\pcurvature}[1]{{\psi_{#1}}}
\newcommand{\pdegeneracy}[1]{{\Delta_{#1}}}
\newcommand{\pkernel}[1]{{\mathscr V({#1})}}

\newcommand{\CartierTransform}[1]{{\mathscr{C}(#1)}}
\newcommand{\Cartier}{\mathbf{C}}

\newcommand{\field}{{\mathsf k}}

\DeclareMathOperator{\End}{End}

\numberwithin{equation}{section}

\addtocounter{section}{0}             
\numberwithin{equation}{section}       

\title[Foliations in positive characteristic]{The space of foliations on projective spaces in positive characteristic}
\author{Wodson Mendson}
\author{Jorge Vit\'orio Pereira}

\date{\today}

\begin{document}

\begin{abstract}
    This work explores the space of foliations on projective spaces over algebraically closed fields of positive characteristic, with a particular focus on the codimension one case. It describes how the irreducible components of these spaces vary with the characteristic of the base field in very low degrees and establishes an arbitrary characteristic version of Calvo-Andrade’s stability of generic logarithmic $1$-forms under deformation.
\end{abstract}

\maketitle
\setcounter{tocdepth}{1}
\tableofcontents

\section{Introduction}

\subsection{Context} The study of the space of holomorphic foliations on projective spaces has been a subject of significant interest since the pioneering work of Jouanolou. Inspired by classical contributions from Jacobi, Darboux, and others, Jouanolou described codimension one foliations of degree one on projective spaces in his book \cite{MR537038}. In particular, his results imply that the space of codimension one holomorphic foliations of degree one on projective spaces, of dimension at least three, has exactly two irreducible components. 

More recently, in \cite{MR3066408}, it was noted that work by Medeiros \cite{MR1842027} provides an analogous description for degree one holomorphic foliations of arbitrary codimension on projective spaces. 

Cerveau and Lins Neto \cite{MR1394970} proved that the space of codimension one holomorphic foliations of degree two on projective spaces of dimension at least three has exactly six irreducible components and described precisely each one of them. The analogous problem for degree two holomorphic foliations of arbitrary codimension was recently studied by Corrêa and Muniz \cite{arXiv:2207.12880}. However,  a complete description of the irreducible components of the space of degree two holomorphic foliations  of codimension $q$  on $\mathbb P^n_{\C}$ remains an open question for any $q$ strictly between $1$ and $n-1$. The situation for degree three foliations is even less understood. While a partial classification exists for codimension one foliations, see \cite{MR4354288}, the problem remains unexplored for higher codimensions.

\subsection{Purpose and discussion of the main results} This work aims to initiate the study of the space of foliations on projective spaces over algebraically closed fields of positive characteristic, with a focus on the codimension one case. Our initial motivation was to understand how the geometry of these spaces vary with the characteristic of the base field, though our perspective evolved throughout the study and we were led to an analog  of Calvo-Andrade's stability of generic logarithmic $1$-forms under deformation in arbitrary characteristic, see Theorem \ref{THM:Log} below.

For codimension one foliations of degree zero and one, we provide a complete description of all irreducible components in arbitrary characteristic. For degree zero foliations, there is always a single irreducible component, see Theorem \ref{T:deg0} which also describes degree zero foliations of arbitrary codimension. Nevertheless, the behavior in characteristic two is different from other characteristics. Specifically, when the  characteristic is different from two, the space of degree zero foliations of codimension one can be identified with the Grassmannian of lines in $(\mathbb P^n_{\field})^{\vee}$, whereas in characteristic two, it corresponds to the projective space  $\mathbb P H^0(\mathbb P^n_{\field} , \Omega^1_{\mathbb P^n_{\field}}(2))$. 

For degree one foliations, the number of irreducible components depends on the characteristic. In characteristic two, there is a single  irreducible component, whereas for higher characteristics, there are exactly two. Moreover, the description is uniform when the characteristic is at least $5$, see Theorem \ref{T:deg1}.

A key technical tool in our analysis of foliations of degree zero and one is Medeiros' description of locally decomposable integrable differential forms with homogeneous coefficients of degree one, see \cite{MR1842027}. Although Medeiros originally formulated his result in the complex setting, the result remains valid in arbitrary characteristic. 

For degree two foliations, our results are less definitive. We are able to show that each of the six irreducible components that exist in characteristic zero specialize to irreducible components in characteristic $p$, when $p$ is at least $5$.  We achieve this by combining \cite[Theorem B]{pereira2025distributionslocallyfreetangent}  with the main result of this paper, namely:

\begin{THM}\label{THM:Log}  
    Let $r\ge 2$ and  $1 \le d_1 \le \ldots \le d_r$  be positive integers. Let $\field$ be an algebraically closed field of characteristic $p>0$.
    Assume that one of the following conditions holds:
    \begin{enumerate}
        \item $r=2$ and none of the integers $d_1,d_2,$ and $d_1 + d_2$ is divisible by $p$; or
        \item $r=2$ and both $d_1$ and $d_2$ are divisible by $p$; or
        \item  $r>2$ and the set  
        \[
            S = \{ i \in \{ 1, \ldots, r\} \, \vert \, p \text{ divides } d_i \}
        \]
        has cardinality different from $r-1$ and $r-2$.
    \end{enumerate}
    Then the set $\Log{d_1, \ldots, d_r}{n}{\field}$ is an irreducible component of $\Fol{d}{n}{\field}$ (with its reduced scheme structure) for $n\ge 3$ and $d = \sum_{i=1}^r d_i - 2$.
\end{THM}

For the definition of the scheme $\Fol{d}{n}{\field}$ we refer to  Section \ref{SS:Fol}, while  the set $\Log{d_1, \ldots, d_r}{n}{\field}$ is defined in Section \ref{SS:Log1}. 

Over the complex numbers, Theorem \ref{THM:Log} was established by Calvo-Andrade \cite{MR1286897}, with a proof based on complex-geometric and topological arguments. For an alternative approach using algebraic methods, which actually proves that $\Fol{d}{n}{\mathbb{C}}$  is generically reduced along $\Log{d_1, \ldots, d_r}{n}{\mathbb{C}}$, see \cite{MR3937325}.

Theorem~\ref{THM:Log} offers yet another proof of Calvo-Andrade's
result, this time from a more arithmetic perspective, see Corollary~\ref{C:CalvoAndrade}.

\subsection{Acknowledgments}  
We thank the anonymous referees for their careful reading and for suggestions that improved the exposition. We also thank João Pedro dos Santos for his interest and comments on a preliminary version of this work. J.V.P.  acknowledges support from  the CAPES-COFECUB (project Ma1017/24), CNPq (PQ Scholarship 304690/2023-6),  Projeto Universal 408687/2023-1 “Geometria das Equações Diferenciais Alg\'ebricas”), and FAPERJ (Grant number E26/200.550/2023). W. Mendson acknowledges support from INCTmat/Brazil, CNPq grant number 170249/2023-9.

\section{Basic definitions}

\subsection{Foliations and distributions} 
Let $X$ be a smooth algebraic variety defined over an arbitrary field  $\field$. 
A foliation $\F$ on $X$ consists of a pair $(T_{\mathcal F}, N^*_{\F})$ of coherent subsheaves of $T_X$ and $\Omega^1_X$ such that $T_{\F}$ is the annihilator of $N^*_{\F}$, $N^*_{\F}$ is the annihilator of $T_{\F}$, and $T_{\F}$ is closed under the Lie bracket. The sheaf $T_{\F}$ is the tangent sheaf of $\F$ and the sheaf $N^*_{\F}$ is the conormal sheaf of $\F$. The rank of $T_{\F}$ is the dimension of $\F$, while the rank of $N^*_{\F}$ is the codimension of $\F$. Since $T_{\F}$ and $N^*_{\F}$ are annihilators of each other, we have that $\dim \F + \codim \F= \dim X$. 

Dropping the closedness under Lie brackets, one obtains the definition of a distribution. To wit, a distribution $\mathcal D$ on $X$ consists of a pair $(T_{\mathcal D}, N^*_{\mathcal D})$ of subsheaves of $T_X$ and $\Omega^1_X$ such that $T_{\mathcal D}$ is the annihilator of $N^*_{\mathcal D}$ and vice-versa. 

\subsection{Differential forms}
If $\mathcal D$ is a distribution of codimension $q$,  taking the determinant of the inclusion $N^*_{\mathcal D} \to \Omega^1_X$ and tensoring the result by $N = \det(N^*_{\mathcal D})^*$ yields a $q$-differential form with coefficients in $N$, that is 
\[
    \omega \in H^0(X, \Omega^q_X \otimes N)
\]
satisfying the following conditions:
\begin{enumerate}
    \item\label{I:cod2} The zero locus of $\omega$ has codimension at least two; 
    \item\label{I:LD} At any closed point $x$ outside the zero locus of $\omega$, and in any local trivialization of $N$, the germ of the differential $q$-form $\omega$ at $x$ can be written as the product of $q$ germs of $1$-forms, $\omega_1, \ldots, \omega_q$;
    \item\label{I:integrable} Moreover, if $\mathcal D$ is not only a distribution but also a foliation then the germs of $1$-forms of the previous item are such that 
    \[
        d \omega_j \wedge \omega_1 \wedge \cdots \wedge \omega_q = 0 
    \]
    for any $j \in \{ 1, \ldots, q\}$.
\end{enumerate}
The twisted $q$-form $\omega$ is not unique, but is unique up to multiplication by elements of $H^0(X,\mathcal O_X^*)$, that is, nowhere-vanishing regular functions. 

We define saturated $q$-forms as those satisfying Condition \eqref{I:cod2}; 
locally decomposable $q$-forms as those satisfying Condition \eqref{I:LD}; and 
integrable $q$-forms as those satisfying Conditions \eqref{I:LD} and \eqref{I:integrable}.
In Medeiros' original terminology, locally decomposable forms were referred to as locally decomposable off the singular set (LDS) forms. We adopt the shorter term locally decomposable for the sake of brevity.

Conversely, given a nonzero twisted $q$-form $\omega \in H^0(X, \Omega^q_X \otimes N)$,  the kernel of the natural  morphism induced by contraction with $\omega$
\[
    \omega^{\flat}  : T_X \to \Omega^{q-1}_X \otimes N
\]
is a saturated subsheaf of $T_X$. Consequently, $(\ker \omega^{\flat}, \Ann(\ker \omega^{\flat})) = (T_\mathcal D, N^*_{\mathcal D})$  for some codimension $q$ distribution $\mathcal D$. 

\subsection{Plücker's equations}
As noted by Medeiros in \cite{MR1842027}, see also \cite[Appendix B]{pereira2025distributionslocallyfreetangent}, a twisted $q$-form $\omega \in H^0(X, \Omega^q_X \otimes N)$ is locally decomposable if, and only if, the morphism
\begin{align*}
    \bigwedge^{q-1} T_X &\longrightarrow \Omega_X^{q+1} \otimes N^{\otimes 2} \\
    v &\longmapsto i_v \omega \wedge \omega
\end{align*}
is identically zero. Notice that this condition can be rephrased as the classical Plücker conditions 
\[
    i_v \omega \wedge \omega = 0 \text{ for any } v \in \bigwedge^{q-1} T_X \, .
\]

Similarly, a $q$-form $\omega$  is integrable if, and only if, it is locally decomposable and the morphism
\begin{align*}
    \bigwedge^{q-1} T_X &\longrightarrow \Omega_X^{q+2} \otimes N^{\otimes 2} \\
    v &\longmapsto i_v \omega \wedge d \omega
\end{align*}
is identically zero. Here and above, when $q=1$, we adopt the convention that $\wedge^0 T_X = \mathcal O_X$ with interior product with elements of $\Omega_X^{\bullet}$ given by the usual multiplication. Hence, for $q=1$, the integrability condition is the usual $\omega \wedge d \omega =0 $ as a section of $\Omega^3_X \otimes N^{\otimes 2}$.

\begin{remark}
    The formulation of Plücker's condition above is independent of the characteristic. However, note  that in characteristic zero, the local decomposability of a $2$-form is equivalent to the condition $\omega \wedge \omega =0$. This equivalence fails in characteristic two. For instance, the $2$-form $\omega = dx_0 \wedge dx_1 + dx_2 \wedge dx_3$ is not decomposable in any characteristic, yet $\omega\wedge \omega = 2 dx_0 \wedge dx_1 \wedge dx_2 \wedge dx_3$, making $\omega \wedge \omega =0$ equal to zero in characteristic two. 
\end{remark}

\subsection{The schemes  $\Dist[q]{d}{n}{\field}$ and $\Fol[q]{d}{n}{\field}$}\label{SS:Fol}
Let $\field$ be an arbitrary algebraically closed field. Let $\mathcal D$ be a codimension $q$ distribution  on the  projective space $\mathbb P^n_{\field}$, $n\ge q+1$.
Let $\omega \in H^0(\mathbb P^n_{\field}, \Omega^q_{\mathbb P^n_{\field}}\otimes N)$ be a twisted $q$-form with zero set of codimension at least two whose kernel
equals $T_{\mathcal D}$. The degree of $\mathcal D$ is, by definition, the
degree of the zero locus of $i^*\omega \in H^0(\mathbb P^q_{\field}, \Omega^q_{\mathbb P^q_{\field}} \otimes i^* N)$, where $i: \mathbb P^q_{\field} \to \mathbb P^n_{\field}$ is a general linear embedding. Hence, if $d = \deg(\mathcal D)$, then 
\[
    N = \mathcal O_{\mathbb P^n_{\field}}(d+q+1).
\]

It follows from  Euler's sequence, that we can identify $H^0(\mathbb P^n_{\field}, \Omega^q_{\mathbb P^n_{\field}}(d+ q + 1))$ with the vector space of homogeneous polynomial $q$-forms
\[
     \sum_{\substack{I = (i_1, \ldots, i_q) \\  0\le i_1 < \cdots < i_q \le n}} a_I(x_0, \ldots, x_n) dx_I
\]
which are annihilated by
the radial vector field $R = \sum_{i=0}^n x_i \frac{\partial}{\partial x_i}$, and have  coefficients $a_0, \ldots, a_n$ of degree $d+1$.
We will say that a homogeneous polynomial $q$-form on $\mathbb A^{n+1}_{\field}$ representing an element of $H^0(\mathbb P^n_{\field}, \Omega^q_{\mathbb P^n_{\field}}(d+q+1))$ is a projective $q$-form of degree $d+q+1$.

The space of codimension $q$  foliations of degree $d$ on the projective space $\mathbb P^n_{\field}$ is, by definition, the locally closed subscheme $\Fol[q]{d}{n}{\field}$ of $\mathbb P H^0(\mathbb P^n_{\field}, \Omega^q_{\mathbb P^n_{\field}}(d+q+1))$
defined by the conditions
\[
    [\omega] \in \Fol[q]{d}{n}{\field} \text{ if, and only if, } 
    \left\lbrace \begin{array}{l}   \codim \sing(\omega) \ge 2, \\ i_v \omega \wedge \omega = 0 \text{ for every  } v \in \bigwedge^{q-1} {\mathbb A^{n+1}}, \text{ and } \\ i_v \omega \wedge d \omega = 0 \text{ for every  } v \in \bigwedge^{q-1} {\mathbb A^{n+1}}.  \end{array} \right.
\]

Likewise, the space of codimension $q$  distributions of degree $d$ on the projective space $\mathbb P^n_{\field}$ is, by definition, the locally closed subscheme $\Dist[q]{d}{n}{\field}$ of $\mathbb P H^0(\mathbb P^n_{\field}, \Omega^q_{\mathbb P^n_{\field}}(d+q+1))$
defined by the conditions
\[
    [\omega] \in \Dist[q]{d}{n}{\field} \text{ if, and only if, } 
    \left\lbrace \begin{array}{l}   \codim \sing(\omega) \ge 2, \text{ and } \\ i_v \omega \wedge \omega = 0 \text{ for every  } v \in \bigwedge^{q-1} {\mathbb A^{n+1}}.  \end{array} \right.
\]

\section{Foliations of degree zero}\label{S:degreezero}

\subsection{Euler's relation}\label{SS:Euler}
If $\omega$ is a homogeneous $q$-form of degree $e$ on $\mathbb A^{n+1}_{\field}$, meaning that the coefficients of $\omega$ have degree $e-q$,  then
by Euler's relation, we have
\begin{equation}\label{E:Euler}
    i_R d \omega + d i_R \omega = e \cdot \omega \, .
\end{equation}
Note that in our conventions both $x_i$ and $dx_i$ are assigned degree one.

As observed in \cite[Chapter 1, Proposition 1.4]{MR537038}, when the characteristic of $\field$ does not divide
$d+q+1$, there is a bijection of the set $H^0(\mathbb P^n, \Omega^{q}_{\mathbb P^n}(d+q+1))$ of projective $q$-forms of degree $d+q+1$ and the set of closed homogeneous $(q+1)$-forms of degree $d+q+1$. This bijection is given by the maps 
\[
    \omega \mapsto d \omega \quad \text{and} \quad \beta \mapsto (d+q+1)^{-1} i_R \beta \, .
\]
Furthermore, when $q=1$ and still under the same assumption on  the characteristic of $\field$,  the set of  integrable projective $1$-forms of degree $d+2$ is in bijection with the set of closed homogeneous $2$-forms $\beta$ of  degree $d+2$ satisfying
$\beta \wedge \beta=0$.

For later use, we record a consequence of Euler's formula.

\begin{lemma}\label{L:projectiveclosed}
    Let $\omega \in H^0(\mathbb P^n_{\field}, \Omega^q_{\mathbb P^n_{\field}}(e))$ be a nonzero projective $q$-form of degree $e$ on $\mathbb A^{n+1}_{\field}$. If $\omega$ is closed then the characteristic of $\field$ divides $e$. Moreover, if $e < p(q+1)$ then $\omega$ is exact.
\end{lemma}
\begin{proof}
    Since $\omega$ is projective, Euler's relation \eqref{E:Euler} reads
    \[
        i_R d\omega = e \cdot \omega \, .
    \]
    If $\omega$ is closed, then the left-hand side vanishes and,
    consequently, the characteristic of $\field$ divides $e$ as claimed.

    If $\omega$ is closed then \cite[Theorem 1.3.4]{BrionKumar05} implies
    that
    \[
        \omega = d\beta + \sum_{\substack{I = (i_1, \ldots, i_q) \\ i_1 <
        \cdots < i_q}} a_I^p \,(x_{i_1}^{p-1}dx_{i_1}) \wedge
        (x_{i_2}^{p-1}dx_{i_2}) \wedge \cdots \wedge
        (x_{i_q}^{p-1}dx_{i_q})
    \]
    for a certain homogeneous $(q-1)$-form $\beta$ and some homogeneous
    polynomials $a_I \in \field[x_0,\ldots, x_n]$. The homogeneity of the
    $a_I$'s follows from that of $\omega$ and $\beta$, together with the
    identity $(b_1 + \cdots + b_k)^p = b_1^p + \cdots + b_k^p$.
    If $\omega$ is projective, closed and not exact, then the sum is
    non-zero, so some $a_I$ is non-zero. The projectiveness condition then
    forces $\deg(a_I) \ge 1$, and the corresponding summand has degree
    $p\,\deg(a_I) + qp \ge p(q+1)$. Hence $e \ge p(q+1)$, as wanted.
\end{proof}

\subsection{Foliations defined by closed projective polynomial \texorpdfstring{$1$}{1}-forms}
Lemma \ref{L:projectiveclosed} implies that there are no closed projective $1$-forms of degree $e$ when $p$ does not divide $e$. In sharp contrast, we have the following proposition. 

\begin{prop}\label{P:Closed}
    Let $\field$ be an algebraically closed field of characteristic $p>0$ and
    let $n \ge 3$ and $e \ge 1$ be integers. Consider the subscheme
    \[
        \Closed{pe -2}{n}{\field} := \mathbb P\!\left( \{ \omega \in H^0(\mathbb P^n_{\field}, \Omega^1_{\mathbb P^n_{\field}}(pe)) \, \vert \, d \omega =0 \text{ and } \codim \sing(\omega)\ge 2 \} \right)
    \]
    of $\mathbb P H^0(\mathbb P^n_{\field}, \Omega^1_{\mathbb P^n_{\field}}(pe))$.
    Then $\Closed{pe-2}{n}{\field}$ is non-empty and is an irreducible component
    of $\Fol{pe -2}{n}{\field}$. Moreover, $\Fol{pe-2}{n}{\field}$ is generically
    reduced at a general point of $\Closed{pe -2}{n}{\field}$.
\end{prop}
\begin{proof}
    Let $F \in \field[x_0, \ldots, x_n]$ be the homogeneous polynomial
    \[
        F = \sum_{i=1}^{n} x_{i-1} x_{i}^{pe -1}
    \]
    borrowed from \cite[Section 2]{MR2129679}. The differential $dF$ of $F$ defines a closed (hence integrable) polynomial projective $1$-form
    on $\mathbb A^{n+1}_{\field}$ showing that $\Closed{pe -2}{n}{\field}$ is non-empty. To establish the proposition we will compute the Zariski tangent space of $\Fol{pe -2}{n}{\field}$ at the foliation $\F$ determined by $dF$.

    Collecting coefficients, we can write
    \[
        dF = x_1^{pe-1}\,dx_0 + \sum_{j=1}^{n-1}\!\left(x_{j+1}^{pe-1}
             - x_{j-1}x_j^{pe-2}\right) dx_j - x_{n-1}x_n^{pe-2}\,dx_n .
    \] 
    At a zero of $dF$, the vanishing of the $dx_0$-coefficient forces $x_1=0$.
    Assuming inductively $x_1 = \cdots = x_j = 0$ for some $1\le j \le n-1$,
    the $dx_j$-coefficient reduces to $x_{j+1}^{pe-1}$, forcing $x_{j+1}=0$.
    Thus $dF$ vanishes only along the line $\{x_1 = \cdots = x_n = 0\}$
    through the origin of $\mathbb A^{n+1}_{\field}$, and the foliation $\F$
    determined by $dF$ has singular set of codimension $n$. In particular,
    $\F$ is a closed point of $\Closed{pe-2}{n}{\field}$, which is thus
    non-empty.

    The Zariski tangent space of $\Fol{pe-2}{n}{\field}$ at $\F$ can be
    identified with the vector space
    \[
        \frac{ \{ \eta \in H^0(\mathbb P^n_{\field}, \Omega^1_{\mathbb P^n_{\field}}(pe)) \, \vert \, dF \wedge d \eta = 0\} }{\field \cdot dF } \, .
    \]
    If $\eta$ is a polynomial projective $1$-form of degree $pe$ satisfying
    $dF \wedge d\eta=0$, then since $\codim \sing(dF) =n \ge 3$, de
    Rham-Saito's lemma \cite{MR413155} implies that $d \eta =dF \wedge \beta$
    for some polynomial form $\beta$. Comparing degrees, we deduce that
    $d \eta =0$. It follows that the Zariski tangent space of
    $\Fol{pe-2}{n}{\field}$ at $\F$ coincides with the Zariski tangent space
    of $\Closed{pe-2}{n}{\field}$ at $\F$. Since $\Closed{pe-2}{n}{\field}$ is
    defined by linear equations on the coefficients of $\omega$ and is
    contained in the Zariski closure of $\Fol{pe-2}{n}{\field}$, we conclude
    that $\Closed{pe-2}{n}{\field}$ is an irreducible component of
    $\Fol{pe-2}{n}{\field}$ and that $\Fol{pe-2}{n}{\field}$ is generically
    reduced at $\F$.
\end{proof}

\subsection{Locally decomposable differential forms with linear coefficients}

Below, we recall Medeiros' description of linear locally decomposable and integrable differential forms.

\begin{thm}\label{T:Medeiros}
    Let $n \ge 3$ and $1 \le q \le n-1$ be integers, and let $\omega$ be a homogeneous locally decomposable $q$-form on $\mathbb{A}_{\field}^{n+1}$. If the coefficients of $\omega$ are linear then there exists a linear change of coordinates such that
    \begin{enumerate}[label=(\alph*)]
        \item\label{I:quadratico} $\omega = \alpha \wedge dx_1\wedge \cdots \wedge dx_{q-1}$ for some linear $1$-form $\alpha$ (when $q=1$ this reads simply $\omega = \alpha$); or
        \item\label{I:pullback}  $\omega$ can be written as
        \[
            \sum_{i=0}^q A_i \, dx_0 \wedge \cdots \wedge \widehat{dx_i} \wedge \cdots \wedge dx_{q} \, .
        \]
        for suitable linear polynomials $A_i$. 
    \end{enumerate}
    Moreover, if $\omega$ is integrable then in Case \ref{I:quadratico} $\alpha = df$ for some quadratic polynomial $ f \in \field [x_0, \ldots, x_n]$ and, in Case \ref{I:pullback}, the polynomials $A_i$ belong to $\field [ x_0, \ldots, x_{q}]$. 
\end{thm}

The original statement was formulated over the field $\C$. However, as we proceed to verify, the result remains valid over any algebraically closed field $\field$. The key modifications to Medeiros' argument are presented in the following result.

\begin{lemma} \label{L:dec} 
    Let $\omega$ be a linear locally decomposable and integrable $q$-form on $\mathbb{A}_{\field}^{n+1}$. If $\omega$ is  non-closed,  then there is a linear change of coordinates such that 
    \[
        \omega = \sum_{i=0}^{q}A_idx_{0}\wedge \cdots \wedge \widehat{dx_i} \wedge \cdots \wedge dx_{q} \, ,
    \]
    where $A_{i}$ are linear forms in $\field[x_0,\ldots,x_{q}]$ for all $i$.
\end{lemma}
\begin{proof}
    Since $\omega$ is integrable, it is locally decomposable by definition and 
    we can write
    \[
        \omega=\omega_1\wedge\cdots\wedge\omega_q 
    \]
    with $\omega_1, \ldots, \omega_q$ suitable rational $1$-forms. Moreover, 
    the integrability of $\omega$ implies that
    \[
        d\omega_j=\sum_{\ell=1}^q \theta_{j\ell}\wedge \omega_\ell
    \]
    for suitable rational $1$-forms $\theta_{j\ell}$. Therefore
    \[
        d\omega
        =\sum_{j=1}^q (-1)^{j-1}\omega_1\wedge\cdots\wedge d\omega_j\wedge\cdots\wedge \omega_q
        =\theta\wedge\omega
    \]
    for some rational $1$-form $\theta$. Hence $d\omega$ is a decomposable $(q+1)$-form. Moreover, after a suitable linear change of coordinates, decomposable $(q+1)$-forms with coefficients of degree zero can be written as $dx_{0}\wedge \cdots \wedge dx_{q}$.

    The integrability condition $i_v\omega \wedge d\omega =0$, for any $(q-1)$-vector $v$ implies that 
    \[
        \omega = \sum_{i=0}^{q}A_idx_{0}\wedge \cdots \wedge \widehat{dx_i} \wedge \cdots \wedge dx_{q}
    \]
    where $A_i$ are linear homogeneous polynomials in $\field[x_0,\ldots,x_n]$. Furthermore, since $d \omega = dx_0 \wedge \cdots \wedge dx_q$ by assumption, computing the exterior derivative of $\omega$ in the above form implies that the linear homogeneous polynomials $A_i$ actually belong to $\field[x_0,\ldots,x_q]$ as wanted.
\end{proof}

\begin{proof}[Proof of Theorem \ref{T:Medeiros}]
    Medeiros' argument remains valid over an arbitrary algebraically closed field, except for his use of \cite[Proposition 1.3.1]{MR1842027}, which relies on analytic methods. However, this can be replaced by Lemma \ref{L:dec} above.
\end{proof}

\subsection{Foliations of degree zero} We now have all the necessary ingredients to state and prove  the classification of degree zero foliations in arbitrary characteristic. 

\begin{thm}\label{T:deg0}
    Let $\field$ be an algebraically closed field. Let $n \ge 3$ and $n-1\ge q \ge 1$ be integers. If the characteristic of $\field$ is different from $2$, or if $q>1$,  then $\Fol[q]{0}{n}{\field}$ is isomorphic to the Grassmannian of $(q+1)$-codimensional linear subspaces of $\mathbb P^n_{\field}$, and  foliations
    of degree $0$ are defined by linear projections $\mathbb P^n_{\field} \dashrightarrow \mathbb P^q_{\field}$. If instead the characteristic
    of $\field$ is  $2$ and $q=1$, then $\Fol{0}{n}{\field} = \mathbb P H^0(\mathbb P^n_{\field}, \Omega^1_{\mathbb P^n_{\field}}(2)) = \Closed{0}{n}{\field}$.
    In all cases, $\Fol[q]{0}{n}{\field}$ (with its reduced scheme structure) is a smooth and irreducible algebraic variety.
\end{thm}
\begin{proof}
    Let $\mathcal F$ be a foliation of degree zero, and let $\omega \in H^0(\mathbb P^n_{\field}, \Omega_{\mathbb P^n_{\field}}^q(q+1))$ be a projective $q$-form of degree $q+1$ defining it. 
    We analyze the two cases predicted by Theorem \ref{T:Medeiros}. 
    
    If $\omega$ falls in  Case \ref{I:quadratico} of Theorem \ref{T:Medeiros}, \ie $\omega = df \wedge dx_1 \wedge \cdots \wedge dx_{q-1}$, we can and will assume that $f \in \field[x_q, \ldots, x_n]$. Then, expanding the projectiveness  condition $i_R \omega=0$, we obtain:
    \begin{enumerate}
        \item $i_R df = 0$; and
        \item $i_R ( dx_1 \wedge \cdots \wedge dx_{q-1} )=0$.
    \end{enumerate}
    Since $f$ is quadratic, the first condition implies that the characteristic of $\field$ is $2$. The second condition forces $q=1$, that is $\omega= df$.

    If instead $\omega$ fits in Case \ref{I:pullback} of Theorem \ref{T:Medeiros}, that is
    \[ 
             \sum_{i=0}^q A_i(x_0, \ldots, x_q) dx_0 \wedge \cdots \wedge \widehat{dx_i} \wedge \cdots \wedge dx_q \, ,
    \]
    then  the projectiveness  condition $i_R \omega=0$, together  with the exactness of the Koszul complex, implies that $\omega$ is a scalar multiple of  $i_R(dx_0 \wedge \cdots \wedge dx_q)$.  Consequently, the foliation $\F$ is defined by the kernel of the differential of the linear projection 
    \begin{align*}
        \mathbb P^n_{\field} &\dashrightarrow \mathbb P^q_{\field} \\ 
        (x_0: \ldots: x_n) & \mapsto (x_0: \ldots: x_q) \, .
    \end{align*}
    As such, $\F$ is completely characterized by the indeterminacy locus of this linear projection, which is  the linear subspace $\{ x_{0} = \cdots = x_q =0 \}$ of codimension $q+1$.
    This proves the result when $q\neq 1$ or the characteristic of $\field$
    is different from~$2$.  It remains to treat the case $q=1$ and
    $\mathrm{char}(\field)=2$.

    We first show that $\mathbb{P}H^0(\mathbb{P}^n_{\field},
    \Omega^1_{\mathbb{P}^n_{\field}}(2))= \Closed{0}{n}{\field}$, i.e.,
    that every projective $1$-form of degree~$2$ is closed in
    characteristic~$2$.  Write $\omega=\sum_i a_i\,dx_i$ with
    $a_i=\sum_j A_{ij}x_j$ linear.  The projectivity condition
    $i_R\omega=\sum_i x_i a_i=0$ gives $A_{ii}=0$ for all $i$, and
    $A_{ij}+A_{ji}=0$ for $i\neq j$.  In characteristic~$2$ the latter
    reads $A_{ij}=A_{ji}$, so
    \[
       d\omega = \sum_{i<j}(A_{ij}+A_{ji})\,dx_i\wedge dx_j = 0 \,.
    \]
    Thus $\Fol{0}{n}{\field}\subseteq
    \mathbb{P}H^0(\mathbb{P}^n_{\field},\Omega^1_{\mathbb{P}^n_{\field}}(2))
    =\Closed{0}{n}{\field}\subseteq\Fol{0}{n}{\field}$,
    and all three coincide.

    Moreover, every $\omega\in\Fol{0}{n}{\field}$ is exact: in
    Case~\ref{I:quadratico}, the argument above gives $\omega=df$; in
    Case~\ref{I:pullback}, $\omega$ is a scalar multiple of
    $i_R(dx_0\wedge dx_1)=x_0\,dx_1-x_1\,dx_0$, which in characteristic~$2$
    equals $x_0\,dx_1+x_1\,dx_0=d(x_0x_1)$.
    This completes the proof.
\end{proof}

\begin{remark}
    Our proof does not explicitly address the scheme structure of $\Fol[q]{0}{n}{\field}$. However, the proof of Proposition \ref{P:Closed}  shows that in the exceptional case where $q=1$ and the characteristic of $\field$ is $2$, the scheme $\Fol{0}{n}{\field}$ is generically reduced. In all other cases, we believe that $\Fol[q]{0}{n}{\field}$ is  reduced everywhere. An attentive reader will notice that when $q=1$ and the characteristic is different from $2$, this follows from the proof of Proposition \ref{P:Log1} below. Moreover, it is very likely that the results and arguments of \cite{MR2590385} can be adapted to prove the reducedness of $\Fol[q]{0}{n}{\field}$ when $q>1$.
\end{remark}

\section{Foliations of degree one}

In this section, we focus on foliations of degree one on projective varieties over fields of arbitrary characteristic. Unlike Section \ref{S:degreezero}, where we examined foliations of arbitrary codimension, here we restrict our attention to codimension one foliations.

After reviewing key properties of codimension one foliations in positive characteristic, following \cite{MR4803684}, we introduce two families of irreducible components of the space of foliations: linear pull-back components and logarithmic components with only two distinct residues. These families contain special members of degree one.

We conclude the section with a detailed description of codimension one foliations of degree one in positive characteristic, building on the previous results.

\subsection{The $p$-curvature and its degeneracy divisor for codimension one foliations} 
The main reference for the definitions and results discussed in this subsection is \cite{MR4803684}. 
Let $X$ be a smooth projective variety defined
over an algebraically closed field $\field$ of characteristic $p>0$. Denote by $\Frob: X \to X$ the absolute Frobenius morphism. If $\F$ is a  foliation on $X$, its $p$-curvature  is the morphism of $\mathcal O_X$-modules given by raising vector fields to theirs $p$-th powers:
\begin{align*}
    \pcurvature{\F} : \Frob^* T_{\F} & \longrightarrow N_{\F} \\
    \sum f_i \otimes v_i & \mapsto \sum f_i v_i^p \mod T_{\F} \, .
\end{align*}

A foliation $\F$ is called $p$-closed if $\pcurvature{\F}$ is identically zero. When $\pcurvature{\F}$ is generically surjective, meaning that the image of $\pcurvature{\F}$ has rank equal to the codimension of $\F$,  we say that $\F$ is $p$-dense.

When a foliation $\F$ has codimension one, then it is either $p$-closed or $p$-dense. In the  $p$-dense case, there exist a uniquely determined effective divisor $\pdegeneracy{\F}$ and a uniquely determined codimension two subfoliation $\pkernel{\F} \subset \F$ such that:
\begin{enumerate}
    \item Away from a subset of codimension at least two, the image of $\pcurvature{\F}$ coincides with $N_{\F} \otimes \mathcal O_X(-\pdegeneracy{\F})$; and
    \item The kernel of the $\pcurvature{\F}$ is equal to $\Frob^*T_{\pkernel{\F}}$.
\end{enumerate}

When $X = \mathbb P^n_{\field}$, the degrees of $\F$, $\pkernel{\F}$, and $\pdegeneracy{\F}$ satisfy the relation
\begin{equation}\label{E:pdivisor}
    \deg(\pdegeneracy{\F}) = p(\deg(\F) - \deg(\pkernel{\F}) - 1) + \deg(\F) + 2
\end{equation}
which is a specialization of \cite[Proposition 4.12]{MR4803684}.

\begin{ex}
    Every foliation on $\mathbb P^n_{\field}$ of degree zero is $p$-closed. Indeed, the codimension one case  follows from Equation (\ref{E:pdivisor}) and in higher codimension this follows from Theorem \ref{T:deg0}.
\end{ex}

\subsection{Linear pull-backs}\label{SS:Linear}

Let $\Lin{d}{n}{\field} \subset \Fol{d}{n}{\field}$ denote the reduced subscheme whose closed points correspond to  foliations on $\mathbb P^n_{\field}$  that arise as pullbacks of degree $d$ foliations on $\mathbb P^2_{\field}$ under a linear projection $\pi : \mathbb P^n_{\field} \dashrightarrow \mathbb P^2_{\field}$. It is well-known that for every $d\ge 0$ and  $n \ge 3$, $\Lin{d}{n}{\mathbb C}$ is an irreducible component of $\Fol{d}{n}{\mathbb C}$. Observe that $\Lin{0}{n}{\field}$ corresponds to foliations on $\mathbb P^n_{\field}$ that are defined by a pencil of hyperplanes, hence have singular set of codimension two. As a corollary of Theorem \ref{T:deg0} we obtain that $\Lin{0}{n}{\field}$ is an irreducible component of $\Fol{0}{n}{\field}$ if, and only if, the characteristic of $\field$ is different from two.  Indeed, when the  characteristic of $\field$ is two, a general element of $\Closed{0}{n}{\field}$ has empty singular set when $n$ is odd or a unique singularity when $n$ is even. In both cases, $\Lin{0}{n}{\field}$ cannot be equal to 
$\Closed{0}{n}{\field}$.
 
\begin{prop}\label{P:linearpullback}
    Let $d > 0$. If $\field$ has characteristic $p>0$, then $\Lin{d}{n}{\field}$ is an irreducible component of $\Fol{d}{n}{\field}$ if and only if $d  < p$.
\end{prop}
\begin{proof}
    Let $\mathcal G$ be a general foliation of $\mathbb P^2_{\field}$ of degree $d\ge 1$. Consider the linear projection
    $\pi : \mathbb P^n_{\field} \dashrightarrow \mathbb P^2_{\field}$ given by $$\pi(x_0:\ldots: x_n)=(x_0:x_1:x_2),$$ and let $\F = \pi^* \mathcal G$.

    According to \cite{Wodson}, see also \cite[Theorem 5.10]{MR4803684}, the degeneracy divisor $\Delta_{\G}$ is a reduced divisor. Moreover, Equation \eqref{E:pdivisor} implies it has  degree $d+2 + p(d-1)$.  Due to the explicit form of $\pi$, $\Delta_{\F} = \pi^* \Delta_{\G}$ and $\Delta_{\F}$ is also reduced and of degree $d+2 + p(d-1)$.
    
    If $\Sigma \subset \Fol{d}{n}{\field}$ is an irreducible component containing $\F$, then by  \cite[Section 6]{MR4803684}, 
    for a general foliation $\F'$ in $\Sigma$ the kernel of the $p$-curvature of $\F'$ is a codimension two foliation $\pkernel{\F'}$
    of degree zero. Theorem \ref{T:deg0} implies that in suitable homogeneous coordinates, $\pkernel \F'$ is defined by the linear projection $(x_0: \ldots: x_n) \mapsto (x_0: x_1: x_2)$. 
    Consequently, $\F'$ is defined by a homogeneous polynomial $1$-form $\omega'$ that can be written as
    \[
        \omega' = a dx_0 + b dx_1 + c dx_2
    \]
    where $a,b,c$ are homogeneous polynomials in $\field [x_0, x_1,x_2,\ldots, x_n]$  of degree $d+1$. The integrability condition $\omega'\wedge d\omega'=0$ implies that 
    $(adb - bda) \wedge dx_0 \wedge dx_1 \wedge dx_2 = (adc - cda) \wedge dx_0 \wedge dx_1 \wedge dx_2 = (cdb - bdc) \wedge dx_0 \wedge dx_1 \wedge dx_2 = 0.$    
    Therefore we deduce from  the integrability condition that
    \[
        \frac{\partial }{\partial x_i}\left( \frac{a}{b} \right) = \frac{\partial}{\partial x_i} \left( \frac{a}{c} \right) = \frac{\partial }{\partial x_i}\left( \frac{b}{c} \right) = 0,
    \]
    for every $i \in \{ 3, \ldots, n\}$. 
    Consequently, we obtain the existence of a polynomial $f \in \field[x_0,x_1,x_2,\ldots, x_n]$ such that   $a,b,c \in f \cdot \field[x_0,x_1,x_2,x_3^p, \ldots, x_n^p]$. Since $\omega'$ has singularities of codimension at least two, $f \in \field$.

    Thus, for every multi-index $I = (i_3, \ldots, i_n) \in \mathbb N^{n-2}$ such that $d + 2 > p \cdot |I|$,  there exists a projective $1$-form $\beta_I$ of degree $d + 2 - p \cdot |I|$ on $\mathbb A^3_{\field}$ with variables $x_0, x_1,x_2$  
    such that
    \[
        \omega' = \sum_{I\in \mathbb N^{n-2}}  (x_3^{i_3} \cdots x_n^{i_n})^p  \beta_I .
    \]
    Since any projective $1$-form has degree at least two, we deduce that
    $\omega'$ is a linear pull-back when $d+2 -p < 2$. If instead $d+2-p \ge 2$, then the above expression for $\omega'$ provides a method 
    to construct deformations of $\omega$ that are not linear pull-backs, since $\omega'$ as above satisfies Frobenius integrability condition $\omega' \wedge d \omega'=0$. 
\end{proof}

For later use, we isolate part of the argument used in the proof of Proposition \ref{P:linearpullback} in the form of a lemma.

\begin{lemma}\label{L:pullback}
    Let $\F$ be a codimension one foliation on $\mathbb P^n_{\field}$ where $\field$ is an algebraically closed field of characteristic $p>0$. If $\F$ contains a codimension $q$ foliation of degree zero, $2\le q \le n-1$,  and $\deg(\F) < p$, then there exists a linear projection $\pi: \mathbb P^n_{\field} \dashrightarrow \mathbb P^q_{\field}$ and a foliation $\G$ on $\mathbb P^q_{\field}$ such that $\F = \pi^* \G$. 
\end{lemma}
\begin{proof}
    The argument follows the proof of Proposition~\ref{P:linearpullback}
    with $\mathbb{P}^q_{\field}$ in place of $\mathbb{P}^2_{\field}$.
    Since, by assumption, $\F$ contains a degree zero 
    foliation of codimension $q$, Theorem~\ref{T:deg0} provides coordinates in which
    $\omega' = \sum_{i=0}^{q-1} a_i\,dx_i$ with
    $a_i \in \field[x_0,\ldots,x_n]$.
    The integrability condition forces
    $a_i \in  \field[x_0,\ldots,x_{q-1},\, x_q^p,\ldots, x_n^p]$.
    Since $\deg(\F) < p$, any term involving $x_j^p$ (for $j \ge q$) would
    contribute degree at least $p > d+1$ to $a_i$, which is impossible.
    Hence $a_i \in \field[x_0,\ldots,x_{q-1}]$ and $\omega'$ is the
    pullback of a form on $\mathbb{P}^q_{\field}$.
\end{proof}

As a consequence of Proposition \ref{P:linearpullback} we get the following proposition.

\begin{prop} 
    Let $\field$ be an algebraically closed field of characteristic  $p > 2$. Let $\mathcal{F}$ be a codimension one foliation of degree one in $\mathbb{P}_{\field}^n$ and suppose that $\mathcal{F}$ is not $p$-closed. Then, $\mathcal{F}$ is a linear pull-back of a degree one foliation on $\mathbb{P}_{\field}^2$.
\end{prop}
\begin{proof}
    Since $\mathcal{F}$ is not $p$-closed, the degeneracy divisor of its $p$-curvature, $\pdegeneracy{\F}$, has degree $-p \deg(\pkernel{\F})+3$ according to Equation \eqref{E:pdivisor}.
    
    If $\deg(\pkernel{\F})  = 0$ then, arguing as in the proof of Proposition \ref{P:linearpullback}, we obtain that $\mathcal{F}$ is a linear pull-back of a degree one foliation on $\mathbb{P}_{\field}^2$.
    
    If $\deg(\pkernel{\F}) \ge 1$ then, since $\pdegeneracy{\F}$ is an effective divisor, it follows  that $p = 3$  and that $\pdegeneracy{\F}$ is the trivial divisor. By \cite[Propositions  4.7 and 4.14]{MR4803684}, any projective $1$-form $\omega \in H^0(\mathbb P^n_{\field}, \Omega^1_{\mathbb P^n_{\field}}(3))$ defining $\F$ is closed. Lemma \ref{L:projectiveclosed} implies that $\omega = df$
    for some polynomial $f \in \field[x_0,\ldots, x_n]$. Therefore, $\F$ is $p$-closed, contradicting our assumptions.     Thus, we conclude that $\mathcal{F}$ is necessarily a linear pullback of a degree one foliation on $\mathbb{P}_{\field}^2$.
\end{proof}

\subsection{Numerically linear pullbacks}
Let $\QLin{d}{n}{\field} \subset \Fol{d}{n}{\field}$ denote the reduced subscheme whose closed points correspond to codimension one foliations $\F$ on $\mathbb P^n_{\field}$ of degree $d$ such that $T_{\mathcal F}$ contains $\mathcal O_{\mathbb P^n_{\field}}(+1)^{\oplus n-2}$. 

Over the complex numbers, we have $\QLin{d}{n}{\mathbb C} =\Lin{d}{n}{\mathbb C}$. However, in positive characteristic, the proof of Proposition \ref{P:linearpullback} 
highlights the distinction between these two spaces, leading to the  more precise statement below.

\begin{prop}\label{P:QLin}
    Let $d > 0$. If $\field$ has characteristic $p>0$, then $\QLin{d}{n}{\field}$ is an irreducible component of $\Fol{d}{n}{\field}$. 
    Moreover, $\QLin{d}{n}{\field}= \Lin{d}{n}{\field}$ if and only if $d  < p$.
\end{prop}

We point out that one can deduce that  $\QLin{d}{n}{\field}$  is an irreducible component of $\Fol{d}{n}{\field}$ using \cite[Theorem B and Corollary 4.14]{pereira2025distributionslocallyfreetangent}, as an alternative to the argument involving the $p$-th powers of vector fields. We opt for the latter method because it leverages specific features of positive characteristic, in contrast to the former method.

\subsection{Logarithmic components I}\label{SS:Log1}
Let $(d_1, \ldots, d_r)$ be an ordered  partition of $d+2$, i.e. $1 \le d_1 \le d_2 \le \ldots \le d_r$ and $\sum_{i=1}^r d_i = d+2$.
Set $\Lambda$ as the $\field$-vector space
\[
    \Lambda = \left\{  (\lambda_1, \ldots, \lambda_r ) \in \field^r \,  \vert \, \sum_{i=1}^r d_i \lambda_i  = 0 \right\} \, .
\]
Observe that $\Lambda = \field^r$ if all integers $d_i$ are divisible by $p$, otherwise $\dim \Lambda = r-1$.

Consider the rational map
\begin{align*}
    \Phi_{n,d_1, \ldots, d_r} : \mathbb P(\Lambda) \times \left( \prod_{i=1}^r \mathbb P H^0(\mathbb P^n_{\field},\mathcal O_{\mathbb P^n_{\field}}(d_i)) \right) &\dashrightarrow \mathbb P H^0(\mathbb P^n_{\field}, \Omega^1_{\mathbb P^n_{\field}}(d+2))  \\
     ( [\lambda_1: \ldots: \lambda_r] , [f_1],\ldots, [f_r] ) &\mapsto \left( \prod_{i=1}^r f_i \right) \left( \sum_{i=1}^r \lambda_i \frac{df_i}{f_i} \right) \, .
\end{align*}
Define $\Log{d_1, \ldots, d_r}{n}{\field}$ as the Zariski closure of the image of $\Phi_{n,d_1, \ldots, d_r}$ in the open subset
of $\mathbb P H^0(\mathbb P^n_{\field}, \Omega^1_{\mathbb P^n_{\field}}(d+2))$ formed by projective $1$-forms with singular set of
codimension at least $2$. Note that $\Log{d_1, \ldots, d_r}{n}{\field}$ is empty if all but one of the $d_i$'s are divisible by $p$ since, in this case, the image of $\Phi_{n,d_1, \ldots, d_r}$ will be contained in the closed subset formed by $1$-forms with codimension one zeros.
In every other case, $\Log{d_1, \ldots, d_r}{n}{\field}$ is non empty.

Let us first study the case $r=2$. We adapt the proof of \cite[Theorem 1]{MR2590385}.

\begin{prop}\label{P:Log1}
    Let  $d_1, d_2$ be two positive integers and let $\field$ be an algebraically closed field of characteristic $p \ge 0$.
    If $n\ge 3$ and $p$ does not divide any of the integers $d_1, d_2,$ and $d_1+d_2$ then $\Log{d_1,d_2}{n}{\field}$ is an irreducible component
    of $\Fol{d_1+d_2- 2}{n}{\field}$. Moreover, $\Fol{d_1+d_2-2}{n}{\field}$ is generically reduced, \ie reduced at the general point of $\Log{d_1,d_2}{n}{\field}$.
\end{prop}
\begin{proof}
    Let $f_1, f_2 \in \field[x_0, \ldots, x_n]$ be homogeneous polynomials of degrees $d_1,d_2$, defining smooth hypersurfaces $H_1$ and $H_2$
    that intersect transversely.
    Let $\F \in \Log{d_1,d_2}{n}{\field}$ be the foliation of degree $d = d_1+d_2 -2$ determined by $\omega = i_R df_1 \wedge df_2$.
    The proof of the proposition will follow the same lines of the proof of Proposition \ref{P:Closed}. More specifically,
    it will consist in determining the tangent space of $\Fol{d}{n}{\field}$ at $\F$. For that, it suffices to solve the equation
    \[
        d \omega \wedge d \eta = 0
    \]
    for $\eta$ projective $1$-form of degree $d+2=d_1+d_2$, since $p$ does not divide $d+2$ by assumption and because of that we can go back and forth from this equation to $\omega\wedge d\alpha + \alpha \wedge d \omega$ using Euler's relation, see Section \ref{SS:Euler}.

    First observe that the projectivization $Z$  of the zero set of $df_1 \wedge df_2$ does not intersect the hypersurfaces $H_1$ and $H_2$. Since
    they are ample divisors, it follows that $\dim Z=0$. As we are assuming $n\ge3$, we have that the codimension of the zero set of $df_1\wedge df_2$ is at least three. Therefore, we are in place to apply de Rham-Saito's lemma and write
    \[
        d \eta = \alpha \wedge df_1 + \beta \wedge df_2
    \]
    for homogeneous $1$-forms $\alpha$ of degree $d +2 - d_1 = d_2$ and $\beta$ of degree $d + 2 - d_2 = d_1$. Differentiating this expression
    and taking the wedge product with $df_1$, we deduce that $d\beta \wedge df_1\wedge d f_2=0$. Similarly,  $d\alpha \wedge df_1\wedge d f_2=0$.

    We apply De Rham-Saito's lemma again to write
    \[
        d \beta = \beta' \wedge df_1 + \beta''\wedge df_2
    \]
    with $\beta'$ of degree $d_1 - d_1=0$ and $\beta''$ of degree $d_1 - d_2 \le 0$. It follows that $d \beta=0$. Therefore, there exist
    polynomials $a_1, \ldots, a_n, b \in \field[x_0, \ldots, x_n]$ such that
    \[
        \beta = db + \sum_{i=0}^n a_i(x_0,\ldots, x_n)^p x_i^{p-1} dx_i \, .
    \]
    Since the degree of $\beta$ is not divisible by $p$, the polynomials $a_i$ are equal to zero and we deduce that $\beta$ is exact, \ie $\beta = db$.

    If $d_1=d_2$, then the argument above shows that $\alpha$ is also exact. We deduce that
    \[
        d\eta = da \wedge df_1 + db\wedge df_2
    \]
    for homogeneous polynomials $a$ of degree $d_2$ and $b$ of degree $d_1$. This shows that the Zariski tangent space of $\Fol{d}{n}{\field}$ coincides with the image of the differential
    of $\Phi_{n,d_1,d_2}$ at the point corresponding to $\F$. If instead $d_1 < d_2$ then a variant of the argument above
    shows that $\alpha = da + c df_1$ for some polynomials $a,c \in \field[x_0, \ldots, x_n]$, cf. \cite[Proof of Proposition 2.1]{MR2590385}.
    Anyway, the conclusion is the same: $d\eta = da \wedge df_1 + db\wedge df_2$ as above.

    Since $\Log{d_1,d_2}{n}{\field}$ is irreducible and its Zariski tangent space at $\F$ coincides with that of $\Fol{d}{n}{\field}$, we conclude as in the proof
    of Proposition~\ref{P:Closed} that $\Log{d_1,d_2}{n}{\field}$ is an irreducible
    component of $\Fol{d}{n}{\field}$ and that $\Fol{d}{n}{\field}$ is generically
    reduced at the general point of $\Log{d_1,d_2}{n}{\field}$.
\end{proof}

\begin{remark}
    Let $\field$ be an algebraically closed field of characteristic $p>0$.
    If $d_1,d_2$ are not divisible by $p$ and $d_1+d_2=0 \mod p$ then $\Log{d_1,d_2}{n}{\field}$ is not an irreducible component of
    $\Fol{d}{n}{\field}$ since it is contained in the irreducible component $\Closed{d_1+d_2 -2}{n}{\field}$.
\end{remark}

We will show below, in Theorem \ref{T:Log}, that the subvarieties $\Log{d_1,d_2}{n}{\field}$ when both $d_1$ and $d_2$ are divisible by the characteristic of $\field$ are irreducible components of $\Fol{d_1+d_2 -2}{n}{\field}$.

\subsection{Foliations of degree one}\label{SS:degreeone}
With the necessary tools in place, we can now state and prove the classification of codimension one foliations of degree one in arbitrary characteristic.

\begin{thm}\label{T:deg1}
    Let $\field$ be an algebraically closed field of characteristic $p \ge  0$.
    \begin{enumerate}
        \item\label{I:p=2} If $p = 2$ then $\Fol{1}{n}{\field}$ has exactly one irreducible component: $\Lin{1}{n}{\field}$.
        \item\label{I:p=3} If $p = 3$ then $\Fol{1}{n}{\field}$ has exactly two  irreducible components: $\Lin{1}{n}{\field}$ and  $\Closed{1}{n}{\field}$.
        \item\label{I:p<>23} If $p \notin \{ 2,3\}$ then $\Fol{1}{n}{\field}$ has exactly two irreducible components: $\Lin{1}{n}{\field}$ and $\Log{1,2}{n}{\field}$.
    \end{enumerate}
\end{thm}
\begin{proof}
    As a preliminary remark, note that $\Lin{1}{n}{\field}$ is an irreducible component of $\Fol{1}{n}{\field}$ for every $p \geq 0$. For $p > 0$ this follows from Proposition~\ref{P:linearpullback}. For $p = 0$, it follows by semi-continuity: since $\Fol{1}{n}{\field}$ and $\Lin{1}{n}{\field}$ are defined over $\mathbb{Z}$, upper semi-continuity of the Zariski tangent space dimension implies that its value at a general point of $\Lin{1}{n}{\field}$ in characteristic zero is at most its value in characteristic $p \gg 0$, which equals $\dim \Lin{1}{n}{\field} = \dim \Fol{1}{n}{\field}$ by Proposition~\ref{P:linearpullback}. When $p = 3$, Proposition~\ref{P:Closed} provides an additional irreducible component $\Closed{1}{n}{\field}$. For $p > 3$, Proposition~\ref{P:Log1} also provides an extra irreducible component $\Log{1,2}{n}{\field}$. 

     Let $\omega$ be a projective $1$-form defining a degree one foliation $\mathcal{F}$ and consider the $2$-form $d\omega$. Since $d \omega$ is closed, it is integrable. Therefore, we are in position to apply Theorem \ref{T:Medeiros} to obtain  homogeneous coordinates such that either: \ref{I:quadratico} $d\omega = df\wedge dx_0$ for some quadratic function $f\in \field[x_0,\ldots,x_n]$, or \ref{I:pullback} $d\omega = \sum_{i,j \in \{0,1,2\}} A_{ij}dx_{i}\wedge dx_{j}$ for some linear polynomials $A_{ij} \in \field[x_0,x_1,x_2]$.

     If $p\neq 3$ and $d\omega$ is as  in Case \ref{I:pullback}, then  multiplying its contraction with the radial  vector field by $1/3$ recovers $\omega$, showing that it depends only on the variables $x_0, x_1, x_2$. Thus, it defines a foliation in $\Lin{1}{n}{\field}$. 

     If $p \notin \{2,3\}$ and $d\omega$ is as in  Case \ref{I:quadratico}, then the explicit form of $d\omega$ shows that $\omega$ defines a foliation in  $\Log{1,2}{n}{\field}$. This completes the proof of  Item \ref{I:p<>23}. 

     Now consider the case $p=2$. If $d\omega$ is as in Case \ref{I:quadratico}, then 
     \[
        \omega = 3 \omega = i_R(d\omega) = 2fdx_0 + x_0 df = x_0 df \, .
     \]
     This implies that $\omega$ is either identically zero or that $\omega$ has a zero locus of codimension one. In either case,  $\omega$ does not define a codimension one foliation of degree one. Thus, in characteristic two,  the only irreducible component of $\Fol{1}{n}{\field}$ is $\Lin{1}{n}{\field}$.

     For $p=3$, we already know from Propositions \ref{P:linearpullback} and \ref{P:Closed} that $\Lin{1}{n}{\field}$ and  $\Closed{1}{n}{\field}$ are irreducible components of $\Fol{1}{n}{\field}$. If $d \omega$ is as in Case \ref{I:quadratico} then:
     \[
        0 = 3\omega =  i_{R}d\omega = 2fdx_0-x_0df = -(fdx_0+x_0df) = -d(fx_0)  .
     \]
     This implies that $f$ is a nonzero multiple of $x_0^2$, so  $d\omega = 0$, implying that $\omega$  defines a foliation in $\Closed{1}{n}{\field}$ in this case.

     Finally, we analyze the possibility that $d \omega$ is as in Case \ref{I:pullback}. Since $p=3$, we have $i_R d \omega = 3 \omega =0$. As the coefficients of $d \omega$ have degree one,   we may write
     \[
        d\omega = \alpha \cdot (i_R dx_0 \wedge dx_1 \wedge dx_2)
     \]
     for some $\alpha \in \field$.
     If $\alpha=0$, then $\omega$ defines a foliation in  $\Closed{1}{n}{\field}$. Otherwise, if $\alpha \neq 0$, consider the vector field
     $v = \frac{\alpha}{2} x_0 \frac{\partial }{\partial x_0}$. Then, a straightforward computation shows that 
     \[
        d \omega = d (i_vi_R dx_0 \wedge dx_1 \wedge dx_2)  \, .
     \]
     Therefore, we can write: 
     \[
        \omega = i_vi_R dx_0 \wedge dx_1 \wedge dx_2 + \theta
     \]
     where $\theta$ is a closed polynomial $1$-form with coefficients of degree two annihilated by the radial vector field $R$. 
     By Lemma \ref{L:projectiveclosed}, we conclude that $\theta = df$. 
     The integrability condition $\omega \wedge d \omega = 0 $ implies:
     \[
        0 = df \wedge d \omega = \alpha \cdot df \wedge (i_R dx_0 \wedge dx_1 \wedge dx_2) = \alpha \cdot i_R( df \wedge dx_0 \wedge dx_1 \wedge dx_2 ) \, .
     \]
     Thus $f \in \field[x_0,x_1, x_2, x_3^3, \ldots, x_n^3]$, and consequently
     \[
        \theta = df \in \bigoplus_{i=0}^2 \field[x_0,x_1,x_2] \cdot dx_i \, .
     \]
     It follows that $\omega$ defines a foliation in $\Lin{1}{n}{\field}$ when $p=3$ and $d \omega\neq 0$. This completes the proof. 
\end{proof}

\section{Closed logarithmic \texorpdfstring{$1$}{1}-forms and their residues}\label{SS:residues}

\subsection{Logarithmic $1$-forms}
Let $X$ be a smooth algebraic variety defined over an algebraically closed field $\field$. Let $D= \sum_{i=1}^{\ell} H_i$ be a reduced and effective divisor on $X$, supported on irreducible hypersurfaces $H_1, \ldots, H_{\ell}$ . Following \cite{MR586450}, we denote by $\Omega^1_X(\log D)$ the sheaf of rational $1$-forms $\omega$ on $X$ such that both $\omega$ and $d\omega$ have simple poles contained in the support of $D$. More precisely, for a point $x \in X$, if $h \in \mathcal O_{X,x}$ is a
local equation for $D$, then a germ of rational $1$-form $\omega$ belongs to $\Omega^1_X(\log D)_x$ if and only if both $h \omega$ and $h d\omega$
are germs of regular differential forms.

If $\omega \in \Omega^1_X(\log D)_x$ is a germ of logarithmic $1$-form with poles along $D$, then there exist  $g \in \mathcal O_{X,x}$,  $\eta \in \Omega^1_{X,x}$, and  $f \in \mathcal O_{X,x}$ such that $\restr{g}{H_i}$ is not identically zero for any $i$, and
\[
    g \omega = f \frac{dh}{h}  + \eta \, .
\]
Indeed, if $z_1, \ldots, z_n \in \mathcal O_{X,x}$ form a local system of parameters  such that $z_1$ is not identically zero on any $H_j$, then we can take $g = \frac{\partial h}{\partial z_1}$, see \cite[Section 1]{MR586450}.

\subsection{Residues}
Let $n_i : \normal{H_i} \to H_i$ be the normalization of $H_i$. The (local) residue of $\omega$ along $H_i$ is defined as
the rational function  $n_i^*(f/g)$. This local residue is well-defined, independent of the choices above, and
induces a morphism
\[
    \Res : \Omega^1_X(\log D) \longrightarrow \bigoplus_{i=1}^\ell \Rational{\normal{H_i}} ,
\]
where $\Rational{\normal{H_i}}$ is the sheaf of rational functions on $\normal{H_i}$, the normalization of $H_i$. The kernel of $\Res$ is precisely $\Omega^1_X$.

We follow the terminology of \cite[\href{https://stacks.math.columbia.edu/tag/0CBN}{Tag 0CBN}]{stacks-project}: a normal crossing divisor is a divisor that, locally in the étale topology, is supported on a union of hypersurfaces defined by a subset of a local system of parameters. A strict normal crossing divisor, or a simple normal crossing divisor, is one that locally, in the Zariski topology, is supported on a union of hypersurfaces defined by a subset of a local system of parameters.

\begin{prop}\label{P:Saito}
    Under the assumptions above, if there exists a closed subset $S \subset X$ such that $\codim S \ge 3$ and $\restr{D}{X-S}$ is a normal crossing divisor  on $X-S$, then the image of $\Res$ is contained in
    \[
        \bigoplus_{i=1}^\ell \mathcal O_{\normal{H_i}} \, ,
    \]
    meaning that the residues of $1$-forms in $\Omega^1_X(\log D)$ are regular functions on the normalization of $D$.
\end{prop}
\begin{proof}
    If $\restr{D}{X-S}$ is a simple normal crossing divisor, then for every $x \in X-S$ there exists a local system of parameters
    $z_1, \ldots, z_n$ such that  $\Omega^1_X(\log D)_x = \sum_{i=1}^r  \mathcal O_{X,x} \frac{dz_i}{z_i} + \Omega^1_{X,x}$,  as  can be verified from  \cite[Theorem 2.9]{MR586450} and \cite[Section 10]{MR1193913}. Consequently, the residues of $\restr{\omega}{X-S}$ are regular functions on the normalization of $\restr{D}{X - S}$ by \cite[Theorem 2.9]{MR586450}. Since $S\cap H_i$ has codimension at least two in $H_i$, the same holds for $n_i^{-1}(S)$ in $\normal{H_i}$, as the normalization is a finite morphism. Therefore, the residues of $\omega$ are regular functions on $\normal{D} = \coprod \normal{H_i}$, the normalization of $D$.
    If $\restr{D}{X-S}$ is normal crossing, but not simple normal crossing, then it suffices to observe that the residue morphism commutes
    with étale morphisms. The regularity of the residues then follows from  simple normal crossing case.
\end{proof}

A divisor $D$ satisfying the assumptions of Proposition \ref{P:Saito} is called a normal crossing divisor in codimension two.

\begin{cor}\label{C:é log}
     Let $D = \sum_{i=1}^{\ell} H_i$ be a reduced and effective divisor on $\mathbb P^n_{\field}$ for some algebraically closed field $\field$. If $D$ is a normal crossing divisor in codimension two and $\omega \in H^0(\mathbb P^n_{\field}, \Omega^1_{\mathbb P^n_{\field}}(\log D))$ is a logarithmic $1$-form with poles on $D$, then there exist constants  $\lambda_1, \ldots, \lambda_{\ell} \in \field$  such that, in homogeneous coordinates,
    \[
        \omega = \sum_{i=1}^{\ell}  \lambda_i \frac{dh_i}{h_i} \, ,
    \]
    where $h_i$ are irreducible and reduced homogeneous polynomials defining $H_i$.
\end{cor}
\begin{proof}
    The residues of $\omega$ are regular functions on proper algebraic varieties, hence they are constants $\lambda_1, \ldots, \lambda_{\ell}$. Applying the residue theorem to the restriction of $\omega$ to a general line implies that
    \[
        \sum_{i=1}^{\ell} \lambda_i \deg(H_i) = 0 .
    \]
    Consequently, the homogeneous rational form $\eta = \sum_{i=1}^{\ell} \lambda_i \frac{dh_i}{h_i}$  is annihilated by the radial vector field, hence it defines a section of $H^0(\mathbb P^n_{\field}, \Omega^1_{\mathbb P^n_{\field}}(\log D))$. The difference $\omega - \eta$ is a logarithmic $1$-form with vanishing residues and thus belongs to $H^0(\mathbb P^n_{\field}, \Omega^1_{\mathbb P^n_\field})$. Since $\mathbb P^n_{\field}$ has no global $1$-forms, we conclude that $\omega = \eta$ as claimed.
\end{proof}

\subsection{A characterization of normal crossing in codimension two}
For later use, we record a characterization of divisors that are normal crossing in codimension two.

\begin{lemma}\label{L:snccodimII} 
    Let $X$ be an affine smooth variety defined over an algebraically closed field $\field$.
    Let $h \in \mathcal O_{X}$ be a  regular function, and let $D$ be the divisor defined by $h$. Assume that $D$ is reduced. 
    The following conditions are equivalent:
    \begin{enumerate}
        \item\label{I:divisorD} The divisor $D$ is normal crossing in codimension two.
        \item\label{I:forsome} The subvariety defined by the ideal $I(D)$, which is
        generated by $h$, $v(h)$ for every $v \in H^0(X,T_X)$, and
        \[
            \det \left(
            \begin{array}{cc}
                    v_1^2(h) & v_2(v_1(h)) \\
                    v_1(v_2(h)) & v_2^2(h) \\
            \end{array}
            \right) \quad \text{ for every } v_1,v_2 \in H^0(X,T_X)
        \]
        has codimension at least three.
    \end{enumerate}
\end{lemma}
\begin{proof} 
    We prove the equivalence by considering both directions.

   First, suppose that $D$ is normal crossing in codimension two. Then, there exists a closed subset $S\subset X$ of codimension at least three such that $\restr{D}{X-S}$ is a normal crossing divisor. We may assume that $S$ also contains all loci where three or more irreducible components of $D$ meet, since these have codimension at least three in $X$. Consequently, for every $x \in D-S$, the germ of $D$ at $x$ is either smooth or, in a suitable étale neighborhood, equal to the sum of exactly two smooth irreducible divisors, say $H_1$ and $H_2$, intersecting transversely.
   In the first case, if we take $v \in T_{X,x}$ with value at $x$ not contained in the kernel of $dh$ then $v(h)$ is a unity. In the second case we take $v_1, v_2 \in T_{X,x}$ such that 
    \begin{enumerate}
        \item the value of  $v_1$ at $x$ is transverse to $T_{H_1}(x) := T_{H_1} / \mathfrak m_x T_{H_1}$ and contained in $T_{H_2}(x) := T_{H_2} / \mathfrak m_x T_{H_2}$; and
        \item the value of  $v_2$ at $x$ is transverse to $T_{H_2}(x)$ and contained in $T_{H_1}(x)$.
    \end{enumerate}     
    A simple computation shows that  the determinant
    \[
        \det \left(
            \begin{array}{cc}
                v_1^2(h) & v_2(v_1(h)) \\
                v_1(v_2(h)) & v_2^2(h) \\   
            \end{array}
        \right)
    \]
    is a unit in $\mathcal O_{X,x}$. Consequently, $I(D)_x$ contains a unit for every $x \in X-S$. Thus, the variety defined by $I(D)$ has codimension at least $3$ as wanted.

    To establish the converse direction, we will prove the contrapositive. Suppose that $D$ is reduced and not normal crossing in codimension two.  We will show that $I(D)$ defines a subvariety of codimension at most two in $X$. 
    
    Let $S$ be an irreducible component of the singular locus of $D$ that has codimension one in $D$. Since we assume that $D$ is not normal crossing in codimension two, such a component exists. Consider the local ring $\mathcal{O}_{X, S}$ of regular functions on $X$ along $S$. As $X$ is regular and $S$ has codimension two in $X$, it follows that $\mathcal{O}_{X, S}$ is a regular local ring of dimension two. Let $x,y \in \mathcal{O}_{X, S}$ be minimal generators of the maximal ideal  $\mathfrak{m}_{X,S}$ of $\mathcal{O}_{X, S}$. Since $S$ is contained in the singular locus of $D$, $h$ is an element of $\mathfrak{m}_{X, S}^2$. Writing $h = ax^2+bxy+cy^2$ for some functions $a,b,c \in \mathcal{O}_{X,S}$, we claim that the quadratic form $ax^2+bxy+cy^2$ degenerates along $S$, i.e., $b^2-4ac = 0$ along $S$. Indeed, $D$ is normal  crossing at the generic point $s$ of $S$ if and only if the initial form of $h$ in $\mathfrak{m}_{X,s}^2/\mathfrak{m}_{X,s}^3$ is a product of two nonproportional linear forms over the residue field, which holds if and only if $b^2 - 4ac \neq 0$ at $s$. Since $D$ fails to be normal crossing at $s$ by assumption, we conclude that $b^2-4ac = 0$ along $S$, meaning $S$ is contained in the hypersurface $\{ b^2-4ac=0\}$. On the other hand, for any sections $v_1,v_2 \in T_X$, the determinant
    \[
        \det \left(
            \begin{array}{cc}
                v_1^2(h) & v_2(v_1(h)) \\
                v_1(v_2(h)) & v_2^2(h) \\   
            \end{array}
        \right)
    \]
    has the form $s+(b^2-4ac)t$, where $s\in \mathfrak{m}_{X, S}$ and $t \in \mathcal{O}_{X, S}$. It follows that $I(D)\mathcal{O}_{X, S}$ is contained in the ideal $\mathfrak{m}_{X, S}$, and thus $S$ is contained in the subvariety defined by $I(D)$. Hence, this subvariety has codimension at most two as desired. 
\end{proof}

\begin{cor}\label{C:snc in cod 2 is open}
    Let $X$ be a smooth projective variety and let $T$ be an algebraic variety, both defined over an algebraically closed field $\field$.
    If $\Delta$ is a divisor on $\mathscr X = X \times T$  then the set
    \[
        \{ t \in T \, | \, \Delta_t = \restr{\Delta}{X \times \{ t\}} \text{ is a reduced  normal crossing divisor in codimension two } \}
    \]
    is an open subset of $T$.
\end{cor}
\begin{proof} 
    Since the normal crossing in codimension two condition and the codimension of the subvariety defined by $I(D_t)$ are both local on $X$, we may cover $X$ by finitely many affine open subvarieties and apply Lemma~\ref{L:snccodimII} to each. For each $t\in T$, let $h_t$ be a local equation of $\Delta_t$ and $I(D_t)$ the ideal generated by $h_t$, $v(h_t)$ and
    \[
        \det \left(
               \begin{array}{cc}
                 v_1^2(h_t) & v_2(v_1(h_t)) \\
                 v_1(v_2(h_t)) & v_2^2(h_t) \\
               \end{array}
             \right) \quad \text{ for every } v_1,v_2 \in T_X
    \]
    By Lemma \ref{L:snccodimII}, $D_t$ is normal crossing in codimension two if and only if the subvariety $V_t$ of $X$ defined by $I(D_t)$ has codimension at least three. The result follows from the upper semi-continuity of the function $t\in T \mapsto \dim V_t$, see for instance \cite[Corollary 13.1.5]{MR217086}.
\end{proof}

\section{Logarithmic Components II}\label{SS:LogII}
We return to the study of the logarithmic components of the space of foliations on projective spaces.
Unlike the proof of Proposition \ref{P:Log1}, where we adapted arguments from characteristic zero to positive characteristic,  here we leverage specific properties of foliations over fields of positive characteristic. In particular, we explore fundamental properties of the Cartier transform of foliations, which we recall below.

\subsection{The Cartier transform of a logarithmic foliation}
Let $X$ be a smooth projective variety defined over an algebraically closed field $\field$ of characteristic $p>0$. 
Let $\F$ be a $p$-dense codimension one foliation defined on $X$. By \cite[Proposition 4.4]{MR4803684}, $\F$ is defined by a closed rational $1$-form $\omega$. 
The Cartier transform  $\CartierTransform{\F}$ of $\F$, as introduced in \cite[Section 4.5, Definition 4.17]{MR4803684}, is the distribution given by the kernel of $1$-form obtained by applying the Cartier operator on $\omega$. 

In the particular case that $\omega$ is of the form 
\[
    \omega = \sum_{i=1}^r \lambda_i \frac{df_i}{f_i} \, ,
\]
the result of applying the Cartier operator $\Cartier$ to it is rather easy to compute:
\begin{equation}\label{E:resultCartier}
    \Cartier(\omega) = \sum_{i=1}^r \lambda_i^{1/p} \frac{df_i}{f_i} \, .
\end{equation}

\begin{lemma}\label{L:continhaslog}
    Let $\F$ a codimension one foliation on $\mathbb P^n_{\field}$ defined by a closed logarithmic $1$-form
    \[
         \omega = \sum_{i=1}^r \lambda_i \frac{df_i}{f_i} \, ,
    \]
    where $\lambda_1, \ldots, \lambda_r \in \field$ are constants. 
    Assume that the following conditions hold:
    \begin{enumerate}
        \item for any $i$ there exists a $j$ such that the ratio $\lambda_i / \lambda_j$ does not belong to $\mathbb F_p$.
        \item The polar divisor of $\omega$ is a simple normal crossing divisor. 
    \end{enumerate}
    Then $\deg(\pkernel{\F}) = \deg(\F) - 1$ and $\deg( \pdegeneracy{\F}) = \deg(\F) + 2$. 
\end{lemma}
\begin{proof}
    By \cite[Proposition 4.19]{MR4803684}, the tangent sheaf of $\pkernel{\F}$ coincides with the saturation of $T_{\F} \cap T_{\CartierTransform{\F}}$ inside $T_X$. 
    Under our assumptions, the distribution $\pkernel{\F}$ is defined by the rational $2$-form
    \[
        \theta = \omega \wedge \Cartier(\omega) = \sum_{0\le i < j \le r} \left(\lambda_i \lambda_j^{1/p} - \lambda_j \lambda_i^{1/p} \right) \frac{df_i}{f_i} \wedge \frac{df_j}{f_j} \, .
    \]
    Since the determinant of the normal bundle of $\pkernel{\F}$ satisfies
    \[
        \det(N_{\pkernel{\F}}) = \mathcal O_{\mathbb P^n_{\field}}\left((\theta)_\infty - (\theta)_0\right) = \mathcal O_{\mathbb P^n_{\field}}\left(\deg (\pkernel{\F}) + 3\right)
    \]
    it suffices  to verify that $(\theta)_{\infty} = (\omega)_{\infty}$ and that $\theta$ does not have codimension one zeros. 

    The equality $(\theta)_{\infty} = (\omega)_{\infty}$ follows from the assumption that for any $i$, there exists a $j$ such that $\lambda_i/ \lambda_j \notin \mathbb F_p$ which ensures that 
    $\lambda_i \lambda_j^{1/p} - \lambda_j \lambda_i^{1/p}  \neq 0$. 

    To show that $\theta$ has no codimension one zeros, assume for contradiction that its  zero locus contains a hypersurface $Z$. The hypersurface $Z$ must intersect $(\theta)_{\infty} = (\omega)_{\infty}$ as both are ample divisors. However, because $(\theta)_{\infty}$ is a simple normal crossing divisor, a local computation shows that $\theta$ remains nonzero in a neighborhood of its polar divisor, contradicting the existence of $Z$. This completes the proof. 
\end{proof}

\subsection{Proof of Theorem \ref{THM:Log}}
For the reader's convenience, we recall the statement of Theorem \ref{THM:Log}. 

\begin{thm}\label{T:Log}
    Let $r\ge 2$ and $1 \le d_1 \le \ldots \le d_r$ be positive integers and let
    $\field$ be an algebraically closed field of characteristic $p>0$. Assume that one
    of the following conditions holds:
    \begin{enumerate}
        \item\label{I:log1} $r=2$ and none of the integers $d_1$, $d_2$, and $d_1+d_2$ is divisible
        by $p$; or
        \item\label{I:log2} $r=2$ and both $d_1$ and $d_2$ are divisible by $p$; or
        \item\label{I:log3} $r>2$ and the set of integers $i\in\{1,\ldots,r\}$ such that $p$ divides
        $d_i$ has cardinality different from $r-1$ and $r-2$.
    \end{enumerate}
    Then $\Log{d_1, \ldots, d_r}{n}{\field}$ is an irreducible component of
    $\Fol{d}{n}{\field}$ (with its reduced scheme structure) for $n\ge 3$ and
    $d = \sum_{i=1}^r d_i - 2$.
\end{thm}
\begin{proof}
    Case~(\ref{I:log1}) is the content of Proposition~\ref{P:Log1}. We therefore assume that condition~(\ref{I:log2}) or~(\ref{I:log3}) holds.
    Our assumptions ensure the existence of  $\lambda_1, \ldots, \lambda_r \in \field$ such that
    \[
        \sum_{i=1}^r d_i \lambda_i=0 \, ,
    \]
    with the additional property that $\lambda_i/\lambda_j \notin \mathbb F_p$ for every $i \neq j$. 
    
    Let $f_i \in H^0(\mathbb P^n_{\field}, \mathcal O_{\mathbb P^n_{\field}}(d_i))$ be homogeneous polynomials defining smooth hypersurfaces $H_i$ such that the divisor $D = \sum_{i=1}^r H_i$ is a simple normal crossing divisor. Let $\omega \in H^0(\mathbb P^n_{\field}, \Omega^1_{\mathbb P^n_{\field}}(\log D))$ be the logarithmic  $1$-form defined in homogeneous coordinates by
    \[
        \omega = \sum_{i=1}^r \lambda_i \frac {df_i}{f_i}  \, ,
    \]
    and let $\F$ be the foliation defined by $\omega$.
    Let $\Sigma$ be any irreducible component of  $\Fol{d}{n}{\field}$ containing $\F$. We will show that
    $\Sigma = \Log{d_1, \ldots, d_r}{n}{\field}$.

    By Lemma \ref{L:continhaslog},  the polar divisor $D$ of $\omega$ coincides with the degeneracy divisor of the $p$-curvature $\pdegeneracy{\F}$.
    Furthermore, by \cite[Lemma 6.3]{MR4803684}, there exists an open subset $U\subset \Sigma$, containing $\F$, such that for any  closed point in $U$ corresponding to a foliation $\F'$, the degeneracy divisor $\pdegeneracy{\F'}$ is reduced. Possibly after further restricting $U$, every $\F' \in U$ is $p$-dense (this is an open condition). By \cite[Proposition~4.4]{MR4803684}, $\F'$ is defined by a closed rational $1$-form $\omega'$ whose polar locus is $\pdegeneracy{\F'}$. Since $\pdegeneracy{\F'}$ is reduced and normal crossing in codimension two  (Corollary \ref{C:snc in cod 2 is open}), $\omega'$ is logarithmic with poles along $\pdegeneracy{\F'}$. By Corollary~\ref{C:é log}, this ensures that each $\F' \in U$ belongs to some $\Log{d_1', \ldots, d_{r'}}{n}{\field}$  where $r', d_1' \le  \ldots \le d_{r'}$ satisfy  $\sum_{i=1}^{r'} d_i' = d+2$.

     To conclude the proof, we must show that $r' = r$ and $(d_1', \ldots, d_r') = (d_1, \ldots, d_r)$. On one hand, the polar divisor $\pdegeneracy{\mathcal{F}'}$
    lives in the fixed linear system $|\mathcal{O}_{\mathbb{P}^n_\field}(d+2)|$, and
    the universal family of effective divisors over a linear system is flat (all fibers
    are Cartier divisors with the same Hilbert polynomial). Since the number of
    irreducible components of the fibers of a flat proper family is upper
    semi-continuous, and $\pdegeneracy{\mathcal{F}}$ has $r$ irreducible components,
    we conclude that $r' \le r$ for a general $\mathcal{F}' \in U$. On the other hand,
    consider an inclusion $\varphi:\mathbb P^1_{\field} \to \mathbb P^n_{\field}$ of a
    general line. The number $r$ equals the number of distinct residues of
    $\varphi^* \omega$. Since $\mathcal{F}' \in U$ is defined by a deformation
    $\omega' \in H^0(\mathbb P^n_{\field}, \Omega^1_{\mathbb P^n_{\field}}(\log
    \pdegeneracy{\mathcal{F}'}))$ of $\omega$, the residues of $\varphi^*\omega'$
    deform those of $\varphi^*\omega$, so $r' \ge r$. Therefore $r' = r$.
    
    Finally, each degree $d_j$ corresponds to the set of points in $\mathbb P^1_{\field}$ where $\varphi^*\omega$ has the same residue $\lambda_j$. Since deformations can only decrease the values of each $d_j$, but their sum remains fixed at $d+2$, it follows that  $(d_1', \ldots, d_r') = (d_1, \ldots, d_r)$, as claimed.
\end{proof}

As a corollary, we obtain a new proof of Calvo-Andrade's theorem on
the stability of generic logarithmic $1$-forms, this time from an
arithmetic perspective.

\begin{cor}\label{C:CalvoAndrade}
    Let $r\ge 2$ and $1\le d_1\le \cdots \le d_r$ be positive integers,
    let $n\ge 3$, and set $d=\sum_{i=1}^{r} d_i -2$. Then
    $\Log{d_1,\ldots,d_r}{n}{\mathbb{C}}$ is an irreducible component of
    $\Fol{d}{n}{\mathbb{C}}$.
\end{cor}
\begin{proof}
    Note that both $\Fol{d}{n}{\field}$ and
    $\Log{d_1,\ldots,d_r}{n}{\field}$ are defined over $\mathbb{Z}$.
    Suppose, for a contradiction, that
    $\Log{d_1,\ldots,d_r}{n}{\mathbb{C}}$ is strictly contained in some
    irreducible component
    $W_{\mathbb{C}}\subset\Fol{d}{n}{\mathbb{C}}$. Then there exist a
    finitely generated $\mathbb{Z}$-subalgebra $A\subset\mathbb{C}$ and
    an integral closed subscheme $W_A\subset\Fol{d}{n}{A}$, with
    geometrically irreducible fibers over $\mathrm{Spec}(A)$, whose base
    change to $\mathbb{C}$ recovers $W_{\mathbb{C}}$ and which strictly
    contains $\Log{d_1,\ldots,d_r}{n}{A}$. Since $A$ has characteristic
    zero, the morphism
    $\mathrm{Spec}(A)\to\mathrm{Spec}(\mathbb{Z})$ is dominant, so its
    image contains $\mathrm{Spec}(\mathbb{Z}[1/N])$ for some integer
    $N\ge 1$. For every prime $p\nmid N$ and every closed point of
    $\mathrm{Spec}(A)$ above $p$, the corresponding geometric fiber of
    $W_A$ is an irreducible closed subscheme of
    $\Fol{d}{n}{\overline{\mathbb{F}}_p}$ strictly containing
    $\Log{d_1,\ldots,d_r}{n}{\overline{\mathbb{F}}_p}$. For every $p$
    larger than $\max(d_1,\ldots,d_r,d_1+d_2)$, the hypotheses of
    Theorem~\ref{T:Log} hold (case~(\ref{I:log1}) for $r=2$;
    case~(\ref{I:log3}) for $r>2$, since the set $S$ is empty), and
    therefore $\Log{d_1,\ldots,d_r}{n}{\overline{\mathbb{F}}_p}$ is
    already an irreducible component of
    $\Fol{d}{n}{\overline{\mathbb{F}}_p}$. This is a contradiction, and
    Calvo-Andrade's result follows.
\end{proof}

\section{Foliations of degree two}

\subsection{Foliations defined by Lie algebras of vector fields}
We identify $\mathfrak{pgl}_{n+1}(\field)$, the quotient of $\End(\mathbb{A}^{n+1}_{\field})$ by multiples of the identity,  with the Lie algebra of regular vector fields on $\mathbb{P}^n_{\field}$, \ie $H^0(\mathbb{P}^n_{\field}, T_{\mathbb{P}^n_{\field}})$. 

If the characteristic of $\field$ does not divide $n+1$ then we can concretely, in homogeneous coordinates,  identify $\mathfrak{pgl}_{n+1}(\field)$ with the Lie algebra of vector fields on $\mathbb A^{n+1}_{\field}$ with linear coefficients and zero divergence.  In general, we will think of  $\mathfrak{pgl}_{n+1}(\field)$ as the quotient of the Lie algebra of linear vector fields on $\mathbb A^{n+1}_{\field}$ with linear coefficients modulo multiples of the radial vector field. 

Given a Lie subalgebra $\mathfrak{h} \subset \mathfrak{pgl}_{n+1}(\field)$, we consider the subsheaf of $T_{\mathbb{P}^n_{\field}}$ defined by the image of the evaluation morphism:
\[
   \mathrm{ev}: \mathfrak{h} \otimes \mathcal{O}_{\mathbb{P}^n} \longrightarrow T_{\mathbb{P}^n_{\field}}.
\]
Since, by assumption, $\mathfrak{h}$ is a Lie subalgebra of $\mathfrak{pgl}_{n+1}(\field)$, the image of $\mathrm{ev}$ is an involutive subsheaf of $T_{\mathbb{P}^n}$. Its saturation then defines the tangent sheaf of a foliation, which we denote by $\mathcal{F}(\mathfrak{h})$.

Before stating our next result, we briefly recall the terminology concerning Kupka and non-Kupka singularities that will be used below. Let $\mathcal F$ be a codimension $q$ foliation on a smooth variety $X$, and let
\[
    \omega \in H^0(X,\Omega_X^q\otimes \mathcal N)
\]
be a twisted $q$-form defining $\mathcal F$. A point $x\in \sing(\mathcal F)$
is called a {Kupka singularity} if there exist an open neighborhood $U$ of $x$, a nowhere vanishing local section $\tau \in H^0(U,\mathcal N)$, and a local $q$-form $\eta \in H^0(U,\Omega_X^q)$ such that 
\[
    \omega_{\mid U}=\eta\otimes \tau
\]
and
\[
    \eta(x)=0   \qquad\text{and}\qquad  d\eta(x)\neq 0.
\]
This definition is independent of the choice of local trivialization of $\mathcal N$.  The set of {non-Kupka singularities} of $\mathcal F$ is
the complement of the Kupka set inside $\sing(\mathcal F)$.

\begin{thm}\label{T:subalgebra}
    Let $\mathfrak{h} \subset \mathfrak{pgl}_{n+1}(\field)$ be such that $\mathrm{ev}$ is injective and has saturated image, \ie  $T_{\mathcal{F}(\mathfrak{h})} \simeq \mathfrak{h} \otimes \mathcal{O}_{\mathbb{P}^n_{\field}}$. Let $q$ denote the codimension of $\mathcal{F}(\mathfrak{h})$. If the set of non-Kupka singularities of $\mathcal{F}(\mathfrak{h})$ has codimension at least three, then any irreducible component $\Sigma$ of $\Fol[q]{d}{n}{\field}$ containing $\mathcal{F}(\mathfrak{h})$ is such that a general closed point of $\Sigma$ corresponds to a foliation defined by a Lie subalgebra of $\mathfrak{pgl}_{n+1}(\field)$ of dimension $\dim \mathfrak{h}$.
\end{thm}
\begin{proof}
    We draw on results from \cite{pereira2025distributionslocallyfreetangent},
    which are proved in a characteristic-free setting.

    By \cite[Corollary~4.11]{pereira2025distributionslocallyfreetangent}
    (the restatement of Theorem~B in loc.\ cit.), the set of points in
    $\Fol[q]{d}{n}{\field}$ corresponding to foliations with locally free
    tangent sheaf and non-Kupka singular locus of codimension at least three
    is open. Since $\mathcal{F}(\mathfrak{h})$ belongs to this set by
    hypothesis, every foliation in a sufficiently small open neighbourhood
    $U \subset \Sigma$ of $\mathcal{F}(\mathfrak{h})$ also has locally free
    tangent sheaf and non-Kupka singular locus of codimension at least three.

    By \cite[Corollary~4.14]{pereira2025distributionslocallyfreetangent},
    the isomorphism class of the tangent sheaf is locally constant on $\Sigma$.
    Since $T_{\mathcal{F}(\mathfrak{h})} \simeq \mathfrak{h} \otimes
    \mathcal{O}_{\mathbb{P}^n_{\field}}$ by hypothesis, the tangent sheaf of
    every foliation $\mathcal{F}' \in U$ is isomorphic to $\mathfrak{h}
    \otimes \mathcal{O}_{\mathbb{P}^n_{\field}}$, and in particular is defined
    by a Lie subalgebra of $\mathfrak{pgl}_{n+1}(\field)$ of dimension
    $\dim \mathfrak{h}$.
\end{proof}

\begin{ex}\label{Ex:affine}
    Let $\field$ be an algebraically closed field of characteristic $p >0$. Let $\mathfrak h$ be the subalgebra of $\mathfrak{pgl}_{4}(\field)$ generated by the classes of the linear vector fields
    \begin{align*}
        v_s & = - x_1\frac{\partial}{\partial x_1} - 2 x_2\frac{\partial}{\partial x_2} - 3x_3\frac{\partial}{\partial x_3} \, \\
        v_n &= x_0 \frac{\partial}{\partial x_1} + x_1 \frac{\partial}{\partial x_2} + x_2 \frac{\partial}{\partial x_3} \, .
    \end{align*}
    Note that $[v_s,v_n] = v_n$. 

    The corresponding foliation on $\mathbb P^3_{\field}$ is defined by the projective $1$-form
    \[
        \omega = i_R i_{v_s} i_{v_n} dx_0 \wedge \cdots \wedge dx_3 \, ,
    \]
    where $R$ is the radial vector field. The set of non-Kupka singularities of $\mathcal F(\mathfrak h)$ is, by definition, the zero locus of $d \omega$. An explicit computation shows that 
    \begin{align*}
        \omega &= (-x_0 x_2 x_3+2 x_1^2 x_3-x_1 x_2^2) dx_0 + (-3 x_0 x_1 x_3+2 x_0 x_2^2) dx_1 \\ & + (3x_0^2x_3-x_0x_1x_2  ) dx_2 + (-2x_0^2x_2+x_0x_1^2) dx_3
    \end{align*}
    and that
    \begin{align*}
        d \omega & =  (7x_1x_3 -3x_2^2 ) dx_0 \wedge dx_1 + (-7x_0x_3 - x_1x_2) dx_0 \wedge dx_2  \\ &+  (3x_0x_2 + x_1^2) dx_0 \wedge dx_3  + (5x_0x_2) dx_1 \wedge dx_2 \\ & + (                           -5 x_0 x_1) dx_1 \wedge dx_3  +
        (5 x_0^2) dx_2 \wedge dx_3 \, .
    \end{align*}
    The zero locus of $d \omega$ is then equal to 
    \begin{enumerate}
        \item The line $\{ x_0 = x_1 = 0\}$  when $p=3$; and 
        \item The two points $(1:0:0:0)$ and  $(0:0:0:1)$ when $p=5$; and
        \item The  point $(0:0:0:1)$ when $p \notin \{3,5\}$.
    \end{enumerate}
\end{ex}

Our next result shows that the so-called exceptional component $\Exc{3}{\mathbb C}$ of $\Fol{2}{3}{\mathbb C}$ has natural analogs in all characteristics $p \neq 3$.

\begin{prop}\label{P:exceptional} 
    Let $\field$ be an algebraically closed field of positive characteristic $p \ne 3$. The closure of the orbit through the point in $\Fol{2}{3}{\field}$ determined by Example \ref{Ex:affine} under the action of $\Aut(\mathbb P^3_{\field})$ is an irreducible component $\Exc{3}{\field}$ of $\Fol{2}{3}{\field}$.
\end{prop}
\begin{proof}
   Notation as in Example \ref{Ex:affine}. Let $\mathfrak h$ be the non-abelian Lie algebra of $\mathfrak{pgl}_4(\field)$ generated by the classes of  $v_s$ and $v_n$. If $p\neq 3$ then the induced foliation $\mathcal F(\mathfrak h)$ satisfies the assumptions of Theorem \ref{T:subalgebra}. Therefore, if $\Sigma \subset \Fol[1]{2}{3}{\field}$ is an irreducible component containing $\F(\mathfrak h)$, there exists an open subset $U\subset \Sigma$ containing $\F(\mathfrak h)$ such that every foliation in $U$ is defined by a non-abelian subalgebra $\mathfrak h'$ of $\mathfrak{pgl}_4(\field)$. The image of the Lie-bracket morphism $\wedge^2 \mathfrak h' \to \mathfrak h'$ is generated by a nilpotent element $v_n'$ of $\mathfrak{pgl}_4(\field)$. Perhaps after passing to a smaller open subset of $U$, we may assume that this nilpotent element corresponds to a $4\times 4$ matrix with one-dimensional kernel. Jordan normal form implies that it is conjugated to $v_n$ and hence we can assume that it is equal to $v_n$. Notice also that $v_n$ corresponds to a regular nilpotent element of $\mathfrak{gl}_4(\field)$ (the Lie algebra of $4\times 4$ matrices over $\field$) and, as such, has $4$-dimensional centralizer and the dimension of the corresponding nilpotent orbit is $12$. Hence, $\Sigma$ is birationally equivalent to a $\mathbb P^2$-bundle over the nilpotent orbit through $v_n$ of the adjoint action on $\mathfrak{pgl}_4(\field)$. In particular, $\Sigma$ has dimension $13$ and is equal to the closure of the $\mathrm{PGL}(4)$ orbit through $\mathcal F(\mathfrak h)$.
\end{proof}

\begin{cor}\label{C:exceptional}
    Let $\field$ be an algebraically closed field of positive characteristic $p \ne 3$, and let $n \ge 4$ be an integer. The closure of the orbit through the point in $\Fol{2}{n}{\field}$ determined by a dominant linear pull-back $\mathbb P^n_{\field} \dashrightarrow \mathbb P^3_{\field}$ of Example \ref{Ex:affine} under the action of $\Aut(\mathbb P^n_{\field})$ is an irreducible component $\Exc{n}{\field}$ of $\Fol{2}{n}{\field}$.
\end{cor}
\begin{proof}
    As verified in Example \ref{Ex:affine}, the non-Kupka singularities of the foliation $\mathcal F(\mathfrak h)$ have codimension three. The same property holds true for any  pullback of $\mathcal F(\mathfrak h)$ under a dominant linear projection $\mathbb P^n_{\field} \dashrightarrow \mathbb P^3_{\field}$. By applying  \cite[Theorem B and Corollary 4.14]{pereira2025distributionslocallyfreetangent} we deduce that any foliation sufficiently close to $\mathcal F(\mathfrak h)$ has tangent sheaf isomorphic to $\mathcal O_{\mathbb P^n_{\field}}(+1)^{\oplus n-3}\oplus \mathcal O_{\mathbb P^n_{\field}}^{\oplus 2}$, and therefore contains a subfoliation of degree zero and codimension three. Lemma \ref{L:pullback} then implies that this foliation  must be a linear pull-back of foliation on $\mathbb P^3$. Proposition \ref{P:exceptional} allows us to conclude. 
\end{proof}

Over the complex numbers, there exist exactly two irreducible components of
$\Fol{2}{3}{\mathbb C}$ whose general element corresponds to a foliation with
trivial tangent bundle: $\Exc{3}{\mathbb C}$ and $\Log{1,1,1,1}{3}{\mathbb C}$.
When $\field$ is arbitrary, Theorem \ref{THM:Log} ensures that
$\Log{1,1,1,1}{3}{\field}$ is an irreducible component of $\Fol{2}{3}{\field}$.
If the characteristic of $\field$ is different from $3$, then Proposition
\ref{P:exceptional} shows that $\Exc{3}{\field}$ is also an irreducible
component. Both $\Exc{3}{\field}$ and $\Log{1,1,1,1}{3}{\field}$ are defined over
$\mathbb{Z}$. Since the integrability condition is cut out by polynomial
equations with integer coefficients, $\Fol{2}{3}{\field}$ is a scheme over
$\mathrm{Spec}(\mathbb{Z})$, and the function associating to each prime $p$
the number of irreducible components of $\Fol{2}{3}{\field}$ whose general
element has trivial tangent bundle is constructible on $\mathrm{Spec}(\mathbb{Z})$, see \cite[\href{https://stacks.math.columbia.edu/tag/055B}{Tag 055B}]{stacks-project}.
Since its value at the generic point equals two, it equals two for all but
finitely many primes $p$, showing that the above list is complete for
$p \gg 0$. However, the situation in small characteristic remains an open
question. Similarly, it would be interesting to describe the irreducible
components of the variety/scheme of two-dimensional Lie subalgebras of
$\mathfrak{pgl}_{n+1}(\field)$, akin to the results established for
$\field = \mathbb C$ in \cite[Theorem C]{pereira2025distributionslocallyfreetangent}.
We do not expect an identical classification in positive characteristic, as
the variety of commuting pairs in $\mathfrak{pgl}_{n+1}(\field)$ fails to be
irreducible when the characteristic of $\field$ divides $n+1$, see
\cite{MR4831204}.

\subsection{Synthesis} Our last result summarizes our current understanding of the irreducible components of $\Fol{2}{n}{\field}$ for algebraically closed fields of positive characteristic. It is a straightforward combination of the results previously established throughout this paper. 

\begin{prop}
    Let $\field$ be an algebraically closed field of characteristic $p > 0$. The following subvarieties are irreducible components of the support of the scheme $\Fol{2}{n}{\field}$:
    \begin{enumerate}
        \item When $p=2$: $\Closed{2}{n}{\field}$, $\Log{2,2}{n}{\field}$, $\Log{1,1,1,1}{n}{\field}$,  $\QLin{2}{n}{\field}$, $\Exc{n}{\field}$.
        \item When $p=3$: $\Log{2,2}{n}{\field}$, $\Log{1,1,2}{n}{\field}$, $\Log{1,1,1,1}{n}{\field}$, $\Lin{2}{n}{\field}$.
        \item When $p\ge 5$: $\Log{2,2}{n}{\field}$, $\Log{1,3}{n}{\field}$, $\Log{1,1,2}{n}{\field}$, $\Log{1,1,1,1}{n}{\field}$, $\Lin{2}{n}{\field}$, $\Exc{n}{\field}$,
    \end{enumerate}
    Moreover, if $p \gg 0$, the above list is complete.
\end{prop}
\begin{proof}
    In all cases, $\QLin{2}{n}{\field}$ is an irreducible component by
    Proposition~\ref{P:QLin}. When $p \ge 3$, the inequality $d = 2 < p$ gives
    $\QLin{2}{n}{\field} = \Lin{2}{n}{\field}$.

    The logarithmic components in each list are irreducible components by
    Theorem~\ref{THM:Log}. 

    When $p \ne 3$, the component $\Exc{n}{\field}$ is irreducible by
    Proposition~\ref{P:exceptional} for $n = 3$ and by
    Corollary~\ref{C:exceptional} for $n \ge 4$. When $p = 3$, the foliation of Example~\ref{Ex:affine} has non-Kupka singular locus of codimension two (a line), so Theorem~\ref{T:subalgebra} does not apply. Whether $\Exc{n}{\field}$ remains an irreducible component of $\Fol{2}{n}{\field}$ in characteristic three is an open question.

    When $p = 2$, the component $\Closed{2}{n}{\field}$ is irreducible by
    Proposition~\ref{P:Closed}.

    For completeness when $p \gg 0$: since the integrability condition is cut out  by polynomial equations with integer coefficients, $\Fol{2}{n}{\field}$ is a scheme
    over $\mathrm{Spec}(\mathbb{Z})$, and each component in the list is defined over $\mathbb{Z}$. The function associating to each prime $p$ the number of irreducible components of $\Fol{2}{n}{\field}$ is constructible on $\mathrm{Spec}(\mathbb{Z})$, \cite[\href{https://stacks.math.columbia.edu/tag/055B}{Tag 055B}]{stacks-project}, hence constant for all but finitely many primes. Since its value at the generic
    point equals six by \cite{MR1394970}, the list is complete for $p \gg 0$.
\end{proof}

\bibliography{references}{}
\bibliographystyle{amsplain}

\end{document}